\documentclass[a4paper]{amsart}
\usepackage{tcolorbox}
\usepackage{amsthm}
\usepackage{amsmath}
\usepackage{amsfonts}
\usepackage{amssymb}
\usepackage{bbm}
\usepackage{bbding}
\usepackage[shortlabels]{enumitem}
\usepackage{mathrsfs}
\usepackage{amscd}
\usepackage{tikz-cd}
\usepackage[active]{srcltx}
\usepackage{mathtools}
\usepackage{thm-restate}
\usepackage[colorlinks,final,backref=page,hyperindex]{hyperref}
\usepackage{enumitem}

\def\antish{{\scriptstyle\text{\rm !`}}}

\newtheorem{thm}{Theorem}[section]
\newtheorem{prop}[thm]{Proposition}
\newtheorem{lem}[thm]{Lemma}
\newtheorem{cor}[thm]{Corollary}
\newtheorem{prop-def}[thm]{Proposition-Definition}
\newtheorem{conj}[thm]{Conjecture}
\newtheorem{defn}[thm]{Definition}

\newtheorem{remark}[thm]{Remark}
\newtheorem{exam}[thm]{Example}

\DeclareMathOperator{\HH}{HH}

\DeclareMathOperator{\Hom}{Hom}
\DeclareMathOperator{\im}{im}

\DeclareMathOperator{\Ext}{Ext}
\DeclareMathOperator{\Tor}{Tor}

\newcommand{\nc}{\newcommand}

\nc{\calG}{\mathcal{G}}
\nc{\cok}{\mathrm{coker}}
\nc{\Ann}{\mathrm{Ann}}
\nc{\Tip}{\mathrm{Tip}}
\nc{\NonTip}{\mathrm{NonTip}}
\nc{\rmH}{\mathrm{H}}
\nc{\rmU}{\mathrm{U}}
\nc{\rmHH}{\mathrm{HH}}
\nc{\antishriek}{{\mathord{\textexclamdown}}}
\nc{\Id}{\mathrm{Id}}
\nc{\ID}{\mathrm{ID}}

\nc{\M}{{\mathcal{M}}}
\nc{\B}{{\mathcal{B}}}
\nc{\U}{{\mathcal U}}
\nc{\D}{{\mathcal D}}

\newcommand{\Path}{{\mathcal{P}}}
\newcommand{\V}{{\mathcal{V}}}

\nc{\la}{\langle}
\nc{\ra}{\rangle}
\nc{\bfk}{\mathbf{k}}

\nc{\ot}{\otimes}

\nc{\red}{\color{red}}

\nc{\mbibitem}[1]{\bibitem{#1}} 
\allowdisplaybreaks[4]

\title{A Computation of the Tamarkin--Tsygan Calculus}

\begin{document}
	\author[Chen]{Jun Chen}
	\address{Jun Chen\\
		School of Mathematics\\
		Nanjing University\\
		Nanjing 210093\\
		Jiangsu\\
		PR China}
	\email{mathcj@nju.edu.cn}

	\author[Ruan]{Xiabing Ruan}
\address{Xiabing Ruan\\
	University of Strasbourg}
\email{ruanxiabing@gmail.com}

	\author[Yang]{Jia Yang}
\address{Jia Yang\\
Beijing Normal University}
\email{202221130026@mail.bnu.edu.cn}

\date{\today}

\keywords{Koszul algebras, Hochschild (co)homology, Algebraic Morse Theory, Gerstenhaber bracket, cup and cap product}

\subjclass[2020]{
	16E05, 	
	16E40  	
}
\begin{abstract}
We compute the full Tamarkin--Tsygan calculus of a Koszul algebra whose global dimension exceeds the number of generators. Our results show that even for algebras possessing an economic presentation and agreeable homological properties, the Hochschild (co)homology, as well as the structure of the Tamarkin--Tsygan calculus may exhibit a rather intricate behavior.
\end{abstract}
\maketitle

\section{Introduction}

The protagonist of this article is the algebra $A=\bfk \la x,y,z  \ra/( x^{2} +yx,xz,zy)$, introduced by Iyudu and Shkarin in \cite{I18} as a counterexample to the conjecture in the textbook \cite{PP05} of Polishchuk and Positselski:

\begin{conj}[\cite{PP05}, Section 7]\label{conj:Polishchuk}
    A Koszul algebra with finite global dimension $n$ has at least $n$ generators.
\end{conj}
It has been proved in \cite{I18} that $A$ is Koszul and has global dimension $4$. Clearly, $A$ has $3$ generators. Hence, $A$ serves as a counterexample to Conjecture~\ref{conj:Polishchuk}.

In commutative algebra, the Auslander--Buchsbaum--Serre theorem (\cite{AB57,S56}) states that for Noetherian local rings, regularity is equivalent to having finite global dimension, and in such a case, the global dimension equals to the dimension of the tangent space at the point corresponding to the maximal ideal. While the algebra $A$ is not a local ring, its completion $\hat A=\bfk\la\!\la x,y,z\ra\!\ra/(x^2+yx,xz,zy)$ is, with the same number of generators and the same global dimension as $A$. If the non-commutative analogue of Auslander--Buchsbaum--Serre theorem were true, then $\hat A$ should be ``regular'' in a good sense. However, the fact that the global dimension exceeds the number of generators suggests that such non-commutative analogue is unlikely to hold. This points toward a kind of singularity that the algebra $A$ may possess in the non-commutative setting, if one attempts to use naive generalizations of various notions from commutative algebra. For us, this algebra serves as an interesting nontrivial example for which we compute the structure known as the Tamarkin--Tsygan calculus.

The notion of Tamarkin--Tsygan calculus was originally introduced by Gel'fand, Daletskii, Tamarkin and Tsygan in \cite{GDTl89, TT00, TTs2005}. It encodes a rich compatible structure on the Hochschild (co)homology of an associative algebra. The Tamarkin--Tsygan calculus of $A$ consists of the following data (by convention, the Hochschild cohomology is concentrated in non-positive degrees):

\begin{itemize}
    \item  the Hochschild homology $\HH_*(A)$ and Hochschild cohomology $\HH^*(A)$;
    \item   the cup product $\cup:\HH^*(A)\otimes \HH^*(A)\to \HH^*(A)$ of degree $0$;
    \item     the cap product $\cap:\HH^*(A)\otimes \HH_*(A)\to \HH_{*}$ of degree $0$;
    \item    the Gerstenhaber bracket $[-,-]:\HH^*(A)\otimes \HH^*(A)\to \HH^*(A)$ of degree $1$;
    \item     the Connes' differential $B:\HH_*(A)\to \HH_*(A)$ of degree 1;
\end{itemize}
with appropriate compatibility conditions.

In the commutative case, the Hochschild–Kostant–Rosenberg isomorphism in \cite{HKR62} states that, for a commutative algebra $C$ over a field of characteristic zero,
\begin{itemize}
    \item The Hochschild homology $\HH_*(C)$ is isomorphic to the K\"ahler differential forms of $C$;
    \item  the Hochschild cohomology $\HH^*(C)$ is isomorphic to the poly vector fields of $C$.
\end{itemize}
Later, Calaque proved in \cite{CV10,CRV12} that
\begin{itemize}
    \item The cup product of $\HH^*(C)$ is the wedge product of poly vector fields;
    \item  the cap product is the contraction of differential forms by multi-vector fields, i.e.\ the interior product.
\end{itemize}
It was also discovered and proved by Rinehart and Connes in \cite{C90,GSR63} that
\begin{itemize}
    \item The Connes' differential on $\HH_*(C)$ is the de Rham differential.
\end{itemize}

While certain developments in non-commutative algebra and geometry exist, a precise geometric meaning of the Hochschild (co)homology and the Tamarkin--Tsygan calculus in the non-commutative case remains elusive. Nevertheless, it has been proved by Armenta, Keller in \cite{AK17,AK19} that the Tamarkin--Tsygan calculus is a derived invariant. Later, Tamaroff proved independently in \cite{T21} that it is a homotopy invariant.

In this article, we compute the Tamarkin--Tsygan calculus on the Hochschild (co)homology of the algebra $A$ mentioned above via algebraic Morse theory, hoping to give some guidance for a better understanding of the possible geometry in non-commutative case.
The upshot of our computation can be summarized as follows. The  Gerstenhaber brackets are rather complicated with no apparent pattern to concise description.
Nevertheless, the Hochschild cohomology, equipped with this bracket, contains subalgebras isomorphic to the positive part of $W(0,0)$ (Corollary \ref{cor: sub Witt of HH}) -- the semidirect product  of the Witt algebra with tensor density modules \cite{C1909,KR87,M86} --
whose basis is  $\{L_n, I_n\}_{n\geqslant 1}$
and Lie bracket is given by
$$[L_n, L_m] = (m-n)L_{m+n}, \ [L_n, I_m] = m I_{m+n}.$$
In contrast, the cup products, cap products and the Connes' differentials are more tractable:
All cup products in the Hochschild cohomology of $A$ vanish, except those with the unit, and  the cap products and the Connes' differentials admit simple compact formulas.
To the best of our knowledge, this is the first complete description of the Tamarkin–Tsygan calculus on the Hochschild homology and cohomology of a non-monomial algebra.

The paper is structured as follows.
Section \ref{Sect: preliminaries} sets up basic facts about the algebra $A$, including the  Koszul resolution  as the minimal free resolution and a $\bfk$-basis.
Section \ref{Sect: Hochschild co + homo} determines the Hochschild homology and cohomology of $A$ by providing explicit vector space bases.
Section \ref{Sect: cup and cap} computes the cup and cap products.
In particular, we prove that the cup product on the Hochschild cohomology of $A$ is trivial except for the multiplication by the unit.
Section \ref{Sect: Connes and Gersten}  establishes  homotopy inverses between the Koszul resolutions  and the bar resolution of $A$ via algebraic Morse theory,
  and subsequently use these maps to compute both the Connes' differential and the Gerstenhaber bracket on  its Hochschild (co)homology.
  The explicit formulas for the cap product and the Gerstenhaber bracket are rather involved and have therefore been placed in Appendix \ref{Sect: Result GBracket}.
  Appendix \ref{Sect: proof}   is devoted to the detailed computation of the Hochschild (co)homology of $A$ and the comparison morphisms.
Appendix \ref{Sect: TTCalculus} recalls basic definitions and properties of the Tamarkin--Tsygan calculus.

 Throughout this paper, $\bfk$ is a field of characteristic $0$. All the unadorned tensor products and $\Hom$ spaces are over $\bfk$. We always denote by $A$ the algebra $\bfk\langle x,y,z\rangle / (x^2+yx,xz,zy)$, and $A^e:=A\otimes A^{\mathrm{op}}$ the enveloping algebra of $A$.
 For simplicity of notation, $n_{1,p}$ will denote the sequence $n_1, \dots, n_p.$

\section*{Acknowledgements}

We would like to express our deepest gratitude to Vladimir Dotsenko for suggesting the research topic and for his continuous guidance throughout this project. His advice on mathematics, writing, and other academic aspects is invaluable.
We are also grateful to Guodong Zhou for kindly making the initial connection between the first two collaborators, which make this joint work possible.
We appreciate Bernhard Keller for the valuable discussions and for his helpful suggestions on potential future directions.
This work was supported by the China Scholarship Council (CSC), the China National Postdoctoral Program for  Innovative Talent (BX20240156), and the ANR project HighAGT (ANR-20-CE40-0016), whose financial support is gratefully acknowledged.

\tableofcontents

\section{Preliminaries}\label{Sect: preliminaries}

In this section, we collect some basic facts about the algebra $A$, including a minimal free resolution (the Koszul resolution) and a $\bfk$-basis.

It is shown in \cite{DC17,I18} that the algebra $A$ is Koszul.
Hence the two-sided Koszul resolution (or two-sided Koszul complex) (see \cite[Chapter~3]{LVOperad} for relevant notions) is a minimal resolution of $A$:
\begin{prop}\label{prop: minimal resolution of A}
	The weight components of the Koszul dual coalgebra $A^{\antish }$ of $A$ are given by:
	$$ \begin{array}{l}
		V_0:= A^{\antish  (0)} =\bfk,  \\
		V_1:=  A^{\antish (1)} =\bfk\{x,y,z\}, \\
		V_2:= A^{\antish (2)} =\bfk\{( x+y) x,xz,zy\},  \\
		V_3:=  A^{\antish (3)} =\bfk\{( x+y) xz,xzy\},  \\
		V_4: = A^{\antish (4)} =\bfk \{ (x+y)xzy\}, \\
		V_n:= A^{\antish  (n)}=0 \quad \text{for} \quad n\geqslant 5.
	\end{array}$$
	whence the \textbf{two-sided Koszul resolution} $K_{\bullet}$ 	of	 $A$  is
	\begin{equation*}\label{eq: minimal resolution}
     0 \to A \otimes  V_4 \otimes  A  \xrightarrow{d_4} A \otimes  V_3 \otimes  A \xrightarrow{d_3} A \otimes  V_2 \otimes  A \xrightarrow{d_2} A \otimes  V_1  \otimes  A \xrightarrow{d_1} A\otimes  A \xrightarrow{d_0} A \to 0,  \end{equation*}
	where the differential  is  defined by the following equations:
	$$\begin{array}{l}
		d_4(1\otimes (x+y)xzy  \otimes 1) = (x+y)\otimes xzy \otimes 1 + 1\otimes(x+y)xz \otimes y ,\\
		d_3(1\otimes(x+y)xz \otimes 1)= (x+y)\otimes xz  \otimes 1 - 1\otimes  (x+y)x \otimes z, \\
		d_3(1\otimes   xzy\otimes 1)= x\otimes zy \otimes 1 - 1\otimes xz  \otimes y,\\
		d_2(1\otimes (x+y)x \otimes 1)=(x+y)\otimes x \otimes 1 +1\otimes x \otimes x + 1\otimes y \otimes x,\\
		d_2(1\otimes xz \otimes 1)=x\otimes z \otimes 1 +1\otimes x \otimes z,\\
		d_2(1\otimes zy \otimes 1)=z\otimes y \otimes 1 +1\otimes z \otimes y,\\
		d_1(1\otimes a \otimes 1)= a\otimes 1- 1\otimes a, a=x,y,z, \\
		d_0(1\ot 1) = 1.
	\end{array}$$
\end{prop}

Following \cite{Green99}, we will use Gr\"{o}bner basis theory to construct a $\bfk$-basis of $A$.
\begin{prop}[\cite{DC17}]\label{prop:basisA}
With respect to the degree-lexicographic order where $x > y > z$, the Gröbner basis of $A$ is given by
	$$\mathcal G = \{x y^n x + y^{n+1}x, xz, zy \}_{n\geqslant 0}. $$
\end{prop}

\begin{cor}[\cite{Green99}]\label{cor:k-basis_A}
	The algebra $A$ admits a $\bfk$-basis $\mathcal B$  consisting of the following elements:
		$$ \begin{array}{lll}	
			1,    ~  y^{n_0},  ~  y^{n_0} x,  &  y^{n_0} x y^{n_1}, ~  x, ~ xy^{n_1}, &  \\
			y^{n_0} x y^{n_1} z^{m_1}   \cdots z^{m_p}, & 	y^{n_0} x y^{n_1} z^{m_1}   \cdots z^{m_p}x, & 	y^{n_0} x y^{n_1} z^{m_1}   \cdots z^{m_p} x y^{n_{p+1}},  \\
			x y^{n_1} z^{m_1}   \cdots z^{m_p},  &  x y^{n_1} z^{m_1}   \cdots z^{m_p}x, & x y^{n_1} z^{m_1}   \cdots z^{m_p} x y^{n_{p+1}}, \\
			y^{n_1} z^{m_1}   \cdots z^{m_p}, &  y^{n_1} z^{m_1}   \cdots z^{m_p}x,  &   y^{n_1} z^{m_1}   \cdots z^{m_p} x y^{n_{p+1}}, \\
			z^{m_1}   \cdots z^{m_p}, &  z^{m_1}   \cdots z^{m_p}x,  & z^{m_1}   \cdots z^{m_p} x y^{n_{p+1}},
		\end{array}     $$
			for $p\geqslant 1$ and $ n_0,n_1,\cdots, n_{p+1},$ $m_1,\cdots, m_p  \geqslant 1.$ 
Here,   the notation $z^{m_1}\cdots z^{m_p}$ is shorthand for the expression
	$$ z^{m_1}xy^{n_{2}}z^{m_{2}} \cdots z^{m_{p-1}}xy^{n_p}z^{m_p} $$
in which the symbols a power of $z$, $x$, and a power of $y$  occur cyclically in that order.  
The shorthand notation $z^{m_1}\cdots z^{m_p}$   will be used systematically throughout the paper.
\end{cor}

\section{Hochschild (co)homology}\label{Sect: Hochschild co + homo}

In this section, we will explicitly provide a $\bfk$-basis for the Hochschild homology  and  Hochschild cohomology of $A$.
For the sake of readability, the computational details of this section are deferred to Appendix~\ref{Sect: proof}.

\subsection{Hochschild Homology}\label{Sect: Hochschild homology}

By virtue of Proposition~\ref{prop: minimal resolution of A},  $\HH_\bullet(A)$ is isomorphic to the homology of the chain complex $A \ot_{A^e} K_\bullet .$ Under the isomorphisms $A\otimes_{A^e}(A\otimes V_\bullet\otimes A)\cong A\otimes V_\bullet$, it suffices to compute the homology of the following complex:
$$ 0\to A\otimes V_4\xrightarrow{d_4}A\otimes V_3\xrightarrow{d_3}A\otimes V_2\xrightarrow{d_2}A\otimes V_1\xrightarrow{d_1}A\to 0,$$
where the boundary maps $d_i$, by abusing the notation, are given by the following equations:
$$\begin{array}{l} d_{1} :A\otimes V_{1}\rightarrow A,\quad  a\otimes p\mapsto ap-pa,  p=x,y,z, \\
	d_{2} :A\otimes V_{2}\rightarrow A\otimes V_{1}, \quad
	\begin{cases}
		a\otimes  (x+y)x    \mapsto a( x+y) \otimes x+xa\otimes x+ xa\otimes y \\
		b\otimes  xz  \mapsto bx\otimes z+zb\otimes x\\
		c\otimes zy  \mapsto cz\otimes y+yc\otimes z,
	\end{cases} \\
	d_{3} :A\otimes V_{3}\rightarrow A\otimes V_{2},\quad
	\begin{cases}
		b\otimes  (x+y)xz  \mapsto b( x+y) \otimes xz-zb\otimes  (x+y)x\\
		a\otimes xzy\mapsto ax\otimes zy-ya\otimes xz,
	\end{cases}\\
	d_{4} :A\otimes V_{4} \rightarrow A\otimes V_{3}, \quad
	a\otimes (x+y)xzy \mapsto a( x+y) \otimes xzy+ya\otimes (x+y)xz.
\end{array}$$

The computed set of bases for $\HH_{\bullet}(A)$ is as follows; their proofs can be found in Appendix~\ref{Sect: proof}.

		\begin{prop}\label{prop:HH_1}
			$\HH_1(A)$ has a $\bfk$-basis consisting of the following elements:
\begin{enumerate}[label=(\roman*)]
  \item $\alpha_1(n)  = x^{n-1} \ot x,$
  \item $\beta_1(n)   = y^{n-1} \ot y,$
  \item $\gamma_1(n)  = z^{n-1} \ot z,$
  \item $ \theta_1(n_{1, p};m_{1, p})   =  \sum_{u,v,w}  wu \otimes v, $
\end{enumerate}
where $n, p, n_1, \cdots, n_p, m_1, \cdots, m_p$ are positive integers, and  the sum in $\theta_1$  is taken over all $u,w\in \B, $ and $ v\in \{x,y,z\}$  such that 
$$uvw = xy^{n_1}z^{m_1}\cdots z^{m_p} \in \bfk \langle x,y,z \rangle.$$
\end{prop}

		\begin{prop}\label{prop:HH_0}
			$\HH_0(A)$ has a $\bfk$-basis consisting of the following elements:
			\begin{enumerate}[label=(\roman*).]
				\item$\zeta_0 = 1,$
				\item $  \alpha_0(n) = x^n,$
				\item $  \beta_0(n) = y^{n}, $
				\item $ \gamma_0(n) = z^{n}, $
                \item $\overline\epsilon_{0}(n_{1,p}; m_{1,p})  =\overline{xy^{n_1}z^{m_1}\cdots z^{m_p}},$
			\end{enumerate}
where      $n, p, n_1 \cdots, n_p, m_1, \cdots, m_p$ are positive integers,   and the bars mean that the powers $(n_1,m_1),\cdots,(n_p,m_p)$ are determined up to a cyclic permutation, i.e., 
    $$\overline{xy^{n_1}z^{m_1}\cdots z^{m_p}} = \overline{xy^{n_{i}}z^{m_{i}}\cdots  z^{m_p}xy^{n_1}z^{m_1} \cdots z^{m_{i-1}}}$$ for $i=1,\cdots,p$. 
		\end{prop}

		\begin{prop}
The Hochschild homology $\HH_n(A)$ of $A$ vanishes for $n \neq 0, 1$.
		\end{prop}

\begin{remark} \label{rmk:HH_0} 
    We know that $\HH_0(A) \cong A/[A,A]$, see for example \cite{W19}. So for monomials in $A$, moving one letter from one side to the other does not change their class in $\HH_0(A)$.
      The image of the basis elements of $A$ in Proposition~\ref{prop:basisA} can be rewritten in $\HH_0(A)$ as: 
\begin{itemize}
        \item[(i)]$ xy^{n} = yxy^{n-1}= \cdots = y^{n-1}xy = y^{n}x = (-1)^{n} \alpha_0(n+1) ;$
        \item[(ii)]  $y^{n} = \beta_0(n)$ and $z^n =\gamma_0(n)$; 
        \item[(iii)] $ xy^{n_1}z^{m_1} \cdots z^{m_p}  = y^{n_1}z^{m_1} \cdots z^{m_p}x = y^{n_1-1}z^{m_1} \cdots z^{m_p}xy = \cdots     $
        \item[] $ = z^{m_1} \cdots z^{m_p} x y^{n_1} = z^{m_1-1} xy^{n_2} z^{m_2} \cdots z^{m_p}x y^{n_1}z = \cdots = \overline\epsilon_0(n_{1,p}; m_{1,p} );  $
        \item[(iv)] the image of other basis elements of $A$ are zero in  $\HH_0(A)$.
    \end{itemize}
\end{remark}

\subsection{Hochschild Cohomology}\label{Sect: Hochschild cohomology}

By virtue of Proposition~\ref{prop: minimal resolution of A},   $\HH^{\bullet}(A)$  can be realized as the cohomology of the complex $\Hom_{A^{e}}(K_{\bullet} ,A).$
Under the isomorphism $ \Hom_{A^e}(A\ot V_{\bullet} \ot A, A) \cong \Hom(V_{\bullet}, A),$ it suffices to compute the cohomology of the following complex:
\begin{equation*}\label{eq: cochain complex for cohomology} 0 \to A \xrightarrow{d_1^{*}}  \Hom(V_1, A) \xrightarrow{d_2^{*}}  \Hom(V_2, A)  \xrightarrow{d_3^{*}} \Hom(V_3, A) \xrightarrow{d_4^{*}}  \Hom(V_4, A) \to 0,\end{equation*}
where the differential $d_i^*$ is given by
\begin{itemize}
 \item[(i)]  for $a\in A$, the coboundary $d_1^{\ast}(a) \in  \Hom(V_1, A)$ is defined as
 $$d_1^*(a)  =  \left( \begin{array}{rcl}  x & \mapsto  & xa-ax\\ y & \mapsto & ya-ay \\ z  & \mapsto & za-az,  \end{array}  \right)$$
 \item[(ii)] for $f= \left( \begin{array}{rcl}  x & \mapsto  & a\\ y & \mapsto & b \\ z  & \mapsto & c  \end{array} \right)   \in \Hom(V_1, A),$ the coboundary $d_2^{\ast}(f) \in  \Hom(V_2, A)$ is defined as
 $$d^{*}_2(f) =  \left( \begin{array}{rcl}  (x+y)x  & \mapsto & (x+y)a + ax +bx \\ xz & \mapsto &  xc + az\\ zy & \mapsto & zb + cy,  \end{array} \right)$$
 \item[(iii)] for $f = \left( \begin{array}{rcl}  (x+y)x  & \mapsto & a \\ xz & \mapsto & b \\ zy & \mapsto & c  \end{array} \right) \in  \Hom(V_2, A),$ the coboundary $d_3^{\ast}(f) \in  \Hom(V_3, A)$ is defined as
 $$ d^{*}_3(f)  = \left( \begin{array}{rcl} (x+y)xz & \mapsto &  (x+y)b - az \\ xzy & \mapsto & xc - by, \end{array} \right)  $$
 \item[(iv)] for $f =  \left( \begin{array}{rcl} (x+y)xz & \mapsto & a \\ xzy & \mapsto & b \end{array} \right) \in  \Hom(V_3, A),$  the coboundary $d_4^{\ast}(f) \in  \Hom(V_4, A)$ is defined as
$$d^{*}_4(f) = \big(   (x+y)xzy  \mapsto (x+y)b + ay \big). $$
\end{itemize}

The computed set of bases for $\HH^{\bullet}(A)$ is as follows; their proofs can be found in Appendix~\ref{Sect: proof}.
\begin{prop}\label{prop: HH^0}
    The  degree $0$ Hochschild cohomology $\HH^{0}(A) $ of $A$  is isomorphic to $\bfk$.
\end{prop}

\begin{prop}\label{prop: HH^1}
The Hochschild cohomology  $\HH^{-1}(A)$  has a $\bfk$-basis  consisting of the following elements:
\begin{enumerate}[(\roman*)]
  \item $A^{-1}   = \left( \begin{array}{rcl}  x &  \mapsto & x \\ y  &  \mapsto  &  y \\ z  &  \mapsto  &  0  \end{array} \right),$
  \item $B^{-1}(n) = \left( \begin{array}{rcl}  x  & \mapsto &  y^{n} x \\  y &  \mapsto & y^{n+1} \\ z &  \mapsto  & 0  \end{array} \right),$
  \item $C^{-1}(n)  = \left( \begin{array}{rcl}  x & \mapsto &  y^{n} x \\ y & \mapsto  &    - x y^{n}  \\ z & \mapsto &  zxy^{n-1}  \end{array} \right),   $
  \item $D^{-1}(n,i) = \left( \begin{array}{rcl}  x & \mapsto &   y^{n} x  \\ y & \mapsto  &  - y^{i}xy^{n-i} \\ z & \mapsto & 0  \end{array} \right),  $
  \item $E^{-1}(n_{0,p+1};m_{1,p}) = \left( \begin{array}{rcl}  x & \mapsto & 0 \\ y & \mapsto  & y^{n_0} x y^{n_1}z^{m_1} \cdots z^{m_p}x y^{n_{p+1}-1}\\ z & \mapsto & 0  \end{array} \right), $
  \item $F^{-1}(n_{1,p+1};m_{1,p}) = \left( \begin{array}{rcl}  x & \mapsto & 0 \\ y & \mapsto  & y^{n_1}z^{m_1} \cdots z^{m_p}xy^{n_{p+1}-1} \\ z & \mapsto & 0  \end{array} \right),  $
  \item $G^{-1}(n_{1,p+1};m_{1,p}) = \left( \begin{array}{rcl}  x & \mapsto & 0 \\ y & \mapsto  & x y^{n_1}z^{m_1} \cdots z^{m_p}xy^{n_{p+1}} \\ z & \mapsto & -z x y^{n_1}z^{m_1} \cdots z^{m_p}xy^{n_{p+1}-1}  \end{array} \right), $
  \item $H^{-1}(n_{2,p};m_{1,p}) = \left( \begin{array}{rcl}  x & \mapsto & 0 \\ y & \mapsto  & 0 \\ z & \mapsto & z^{m_1}  \cdots  z^{m_p}  \end{array} \right), $
\end{enumerate}
 where $ 1\leqslant i \leqslant n-1, $ and $n, p,  n_0,n_1,\cdots, n_{p+1}, m_1, \cdots, m_p $ are positive integers.
\end{prop}

\begin{prop}\label{prop: HH^2}
The Hochschild cohomology   $\HH^{-2}(A)$  has a $\bfk$-basis consisting of the following elements:
\begin{enumerate}[(\roman*)]
\item $ A^{-2}  = \left( \begin{array}{rcl}  (x+y)x & \mapsto & x \\ xz & \mapsto & 0 \\ zy & \mapsto & 0  \end{array} \right), $
\item $B^{-2}(n_{2,p};m_{1,p})  = \left( \begin{array}{rcl} (x+y)x & \mapsto  &   z^{m_1}  \cdots z^{m_p}x \\ xz & \mapsto  & 0 \\ zy & \mapsto & 0  \end{array} \right),$
\end{enumerate}
 where $ p, n_2,\cdots, n_p, m_1, \cdots, m_p $ are positive integers.
\end{prop}

\begin{prop}\label{prop: HH^4}
The Hochschild cohomology   $\HH^{-4}(A)$  has a $\bfk$-basis consisting of the following elements:
\begin{enumerate}[(\roman*)]
\item  $A^{-4}  : (x+y)xzy \mapsto  1, $
\item  $A^{-4}(n)  :  (x+y)xzy  \mapsto y^{n}x,  $
\item  $ B^{-4}(n_{0,p};m_{1,p})  : (x+y)xzy  \mapsto  y^{n_0-1}xy^{n_1}z^{m_1} \cdots z^{m_p},    $
\item $C^{-4}(n_{0,p};m_{1, p})  : (x+y)xzy  \mapsto  y^{n_0-1}xy^{n_1}z^{m_1} \cdots z^{m_p}x,    $
\item $D^{-4}(n_{2, p};m_{1, p})  : (x+y)xzy  \mapsto z^{m_1} \cdots z^{m_p},    $
\item $E^{-4}(n_{2, p};m_{1, p})  :    (x+y)xzy  \mapsto  z^{m_1}  \cdots z^{m_p} x,  $
\end{enumerate}
 where $ n,p, n_0,n_1,\cdots, n_{p+1}, m_1,\cdots, m_p$ are positive integers.
\end{prop}

\begin{prop}\label{prop: HH^ greater than 4}
The Hochschild cohomology $\HH^{-n}(A)$ of $A$ vanishes for $n \neq 0,$ $1,2,4$.
\end{prop}

\section{Cup Product and Cap Product} \label{Sect: cup and cap}

\subsection{Cup Product}\label{Sect: Cupproduct}

Following \cite[Section 3]{BLS17}, the cup product on Hochschild cohomology, when restricted to Koszul cochains, takes the following explicit form.
For any Koszul $m$-cochain $f:V_m \to A$ and any Koszul $n$-cochain  $g:V_n \to A,$ their cup bracket is
\begin{equation}\label{eq: cup for Koszul cochain}
f\cup g(a_1\cdots a_{m+n}) : = (-1)^{mn}f(a_1 \cdots a_m) \cdot g(a_{m+1} \cdots a_{m+n}),\end{equation}
 for any $a_1\cdots a_{m+n} \in V_{m+n}.$

\begin{prop}
	$(\HH^{\bullet}(A), \cup) $ is a graded commutative associative  algebra with trivial product except for the unit action.
	\end{prop}
\begin{proof}
	Since $\HH^{-n}(A)$ vanishes for $n=3$ or $n\geqslant 5,$ the cup product vanishes on the component  $\HH^{-m}(A)\otimes \HH^{-n}(A)$ for $(m,n) =(1,2),(2,1),(1,3),(3,1)$ and for all $(m,n)$ such that $m+n > 4$.
Since $\HH^{0}(A) = \bfk$, the cup product involving $\HH^0(A)$ is  simply  multiplication by the unit.
Hence, the only nontrivial cup product to verify is its restriction to $\HH^{-1}(A) \otimes \HH^{-1}(A)$ and $\HH^{-2}(A) \otimes \HH^{-2}(A).$

Consider the cup product that is restricted to $\HH^{-1}(A) \otimes \HH^{-1}(A)$. 
 Due to the graded commutativity of the cup product, we only need to compute $f \cup g$ for basis elements $f$ and $g$ (as given in Proposition~\ref{prop: HH^1}) such that the type of $f$ precedes the type of $g$ 
in the prescribed order.
A direct computation shows that, as cochains (before quotienting by  $\im d_{\bullet}^{*}$), all those cup products  are zero. The only exception is the following product, which is a coboundary:
$$\begin{array}{rl}& C^{-1}(n)\cup F^{-1}(n_{1,p+1}; m_{1,p}) \\ = & \left(\begin{array}{rcl}
			(x+y)x & \mapsto &  0  \\
			xz & \mapsto &  0 \\
			zy & \mapsto &  -zxy^{n-1}\cdot y^{n_1}z^{m_1} \cdots z^{m_p}xy^{n_{p+1}-1}
		\end{array} \right)  \\ =& d_2^*\left(\begin{array}{rcl}
			x & \mapsto & 0 \\
			y & \mapsto & -xy^{n+n_1-1}z^{m_1} \cdots z^{m_p}xy^{n_{p+1}-1} \\
			z & \mapsto & 0 \\
		\end{array}\right).\end{array}$$
Similarly, a direct computation shows that the cup product restricted to $\HH^{-2}(A) \otimes \HH^{-2}(A)$ is also zero. 
	\end{proof}

\subsection{Cap Product}\label{Sect: Capproduct}

Following \cite[Section 4]{BLS17}, the cap product of Hochschild cohomology and homology, when restricted to Koszul cochains and chains, takes the following explicit form.
 For any Koszull $m$-cochain $f$ and any Koszul $n$-chain $w = a_0 \ot a_1\cdots a_n \in A \ot V_n,$ their cap product is
 \begin{equation}\label{eq: cap for Koszul cochain}
w \cap f : =(-1)^{mn} a_0 f(a_1\cdots a_m) \ot a_{m+1}\cdots a_n \in A \ot V_{n-m}.  \end{equation}

Since $\HH_{n}(A)$ vanishes for $n \neq 0,1$, and since $\HH^0(A) = \bfk$ is generated by the unit, the only possibly nontrivial cap product is
$$\cap:  \HH_{1}(A) \ot \HH^{-1}(A) \to \HH_{0}(A).$$
By writing the basis of $\HH_1(A)$ from Proposition~\ref{prop:HH_1}  in the form
$$ n \alpha_1(n) = \sum_{uvw =x^n} wu \ot v, \quad n \beta_1(n) = \sum_{uvw = y^n} wu \ot v, \quad  n\gamma_1(n) = \sum_{uvw= z^n} wu \ot v, $$
$$ \theta_1(n_{1,p}; m_{1,p}) = \sum_{uvw = xy^{n_1}z^{m_1}\cdots z^{m_p}} wu \ot v,$$
the cap product admits the following expressions.
 \begin{prop} The action of  	$\HH^{\bullet}(A)$ on  $\HH_{\bullet}(A)$   induced by the cap product is determined by the following equations.
 For $\Omega =  x^n, y^n, z^n, \text{ or } xy^{n_1}z^{m_1} \cdots z^{m_p}$
 	with $p \geqslant 1, n,n_1,\cdots n_p,$ $m_1,\cdots, m_p\geqslant 1, $ and for $f\in \HH^{-1}(A),$
 	\begin{equation}\label{eq: cap formula}  \sum_{uvw=\Omega} wu \ot v \cap f =  -\sum_{uvw=\Omega}  uf(v)w + \im(d_1). \end{equation}
 \end{prop}
The above formulas are immediate consequences of Equation~\eqref{eq: cap for Koszul cochain} and Remark~\ref{rmk:HH_0}; for better readability,  explicit expressions for the cap product are provided in Appendix~\ref{Sect: Result GBracket}.
We now illustrate the computation through a concrete example. The remaining cases can be handled similarly.
\begin{exam}
Let $f = B^{-1}(n), w = \theta_1(n_{1};m_{1}),$ we have
    $$ \begin{array}{rll}
    w\cap f    &=  - f(x)y^{n_1}z^{m_1} - \sum\limits_{i=1}^{n_1} xy^{i-1}f(y)y^{n_1-i}z^{m_1} - \sum\limits_{i=1}^{m_1} xy^{n_1}z^{i-1}f(z)z^{m_1-i}&  + \im (d_1)\\
           &  = -y^nxy^{n_1}z^{m_1} - n_1 x y^{n_1+n}z^{m_1}    &  + \im (d_1) \\
          &=  -n_1 \overline{\epsilon}_0(n_1+n; m_1),
    \end{array}$$
  where the last equality follows from Remark~\ref{rmk:HH_0}.
\end{exam}

\section{Connes' Differential and Gerstenhaber bracket}\label{Sect: Connes and Gersten}
The cup and cap product can be deduced directly from the Koszul resolution. To get the Connes' differential  and the Gerstenhaber bracket,
we need the comparison morphisms between the Koszul resolution  and the Bar resolution of $A$, which we compute using algebraic Morse theory.

\subsection{Comparison morphisms} \label{Sect: Comparison morphisms}
We begin by recalling the key notations and main results of algebraic Morse theory from  \cite{CLZ24}, with minor modifications for simplicity.

Let $(X_{\bullet}, d_{\bullet})$ be a complex of vector spaces.
Suppose that for each $n\in \mathbb{Z}$,  there exists a decomposition into direct sums of subspaces
$$X_n=\oplus_{i\in I_n} X_{n, i}.$$
So $d_n:X_n\to X_{n-1}$ has a matrix presentation  $d_n=(d_{n, ji})$ with $i\in I_n, j\in I_{n-1}.$
We shall construct a  weighted  quiver  $Q=Q_{X_{\bullet}}$  as follows:
\begin{itemize}
	
	\item[(Q1)]
	The  vertices are the pairs   $(n, i)$ with $ n\in \mathbb{Z}, i\in I_n$;
	
	\item[(Q2)] if a map $d_{n, ji}$ with $i\in I_n, j\in I_{n-1}$ does not vanish, then draw an arrow from $(n, i)$ to $(n-1, j)$;
	
	\item[(Q3)] for an arrow in (Q2), its  weight  is just the map  $d_{n, ji}$.
\end{itemize}

A \textit{partial  matching} is a full subquiver $\M$ of $Q$ such that
\begin{itemize}
	\item[(M1)]each vertex in $Q$  belongs to at most  one arrow of $\M$;
	
	\item[(M2)] each arrow in $\M$  has its weight  invertible as a linear maps.
\end{itemize}

Given a  partial matching  $\M$, we can construct a new weighted quiver $Q^\M$ with additional dotted arrows as follows:
\begin{itemize}
	\item[(QM1)]Keep everything for all  arrows  which are not in $\M$ (they will be called \textit{thick arrows});
	
	\item[(QM2)] For an  arrow in $\M$, replace it by  a new \textit{dotted arrow} in the reverse direction and the weight of this new arrow is the negative  inverse of the weight of the original arrow.
\end{itemize}

A vertex of $Q$   a \textit{critical vertex} (with respect to $\M$), if it is not incident to  any arrow in $\M$, and
 a path in $Q^\M$ is called \textit{zigzag} if dotted arrows and  thick arrows  appear  alternately.
We adopt the following notations:

\begin{center}
\begin{tabular}{| r | l |}

   $\V_n$ & $=\{(n, i)\mid i\in I_n\}$ \\[0.1cm]

   $\U_n$ &  $= \{(n,i) \in \V_n  \mid  (n,i) \xrightarrow{d_{n,ji}}(n-1,j) \in \M \}$ \\[0.1cm]

    $\D_n$ &  $= \{(n-1,j) \in \V_n  \mid  (n,i) \xrightarrow{d_{n,ji}}(n-1,j) \in \M \}$ \\[0.1cm]

    $\V_n^\M$ &  the set of critical vertices \\[0.1cm]

    $\varphi_p^\M$ &   $=w_n \circ w_{n-1} \circ \cdots \circ w_1,$ for $p:\bullet \xrightarrow{w_1} \bullet \cdots \bullet \xrightarrow{w_n} \bullet$   \\[0.1cm]

    $\Path^\M((n, i), (m, j))$  & the set of all zigzag paths from $(n, i)$ to $(m, j)$  in $Q^\M$
\end{tabular}
\end{center}

A \textit{Morse matching} is a partial  matching which satisfies the  \textit{local  finiteness hypothesis} in \cite[Section 3]{CLZ24} called (LFH).

The following Morse condition is frequently used in practical computations.
\begin{prop}[\cite{CLZ24}]\label{Prop: sufficient condition for Morse matching}
	Let $\M$ be a partial matching of $Q$. If any zigzag path from $(n,i)$ is of finite length for each vertex $(n,i)$ in $ Q^{\M},$ then $\M$ is a Morse matching.
\end{prop}

 Given a Morse matching $\M$, we can construct a new complex (called \textit{Morse complex}) $( {X}_{\bullet}^\M, d_{\bullet}^\M)$ as follows:

The complex ${X}_{\bullet}^\M$ has its $n$-th component $X_n^\M=\oplus_{(n, i)\in \V_n^\M} X_{n, i}$ and the differential
$d_n^\M: X_n^\M\to X_{n-1}^\M$ has the matrix presentation
$d_n^\M=(d_{n, ji}^\M)$ with $(n, i)\in \V_n^\M, (n-1, j)\in \V_{n-1}^\M$ and where
$d_{n, ji}^\M: X_{n, i}\to X_{n-1, j} $  is defined to be
$$d_{n, ji}^\M=\sum_{p\in  \Path^\M((n, i), (n-1, j))} \varphi^\M_p.$$

The main theorem of algebraic Morse theory is as follows.
\begin{thm}[\cite{CLZ24}]\label{Thm: main result of algebraic Morse theory}
	\begin{itemize}
		
		 \item[(i)] Within the above setup,
		$({X}_{\bullet}^\M, d_{\bullet}^\M)$ is a complex.

		 \item[(ii)]   Define maps \begin{align*}
			f_n: X_n^\M & \rightarrow X_n \\
			x \in X_{n,i} & \mapsto f_n(x):=x+ \sum_{(n, j)\in \U_n}\sum_{p\in  \Path^\M   ((n,i), (n, j))} \varphi^\M_p(x),
		\end{align*}
		and
		\begin{align*}
			g_n: X_n & \rightarrow X^\M_n \\
			x \in X_{n,i} & \mapsto g_n(x):=\left\{\begin{array}{ll} \sum\limits_{(n, j)\in \V^\M_n}\sum\limits_{p\in   \Path^\M   ((n,i), (n, j)) } \varphi^\M_p(x),& (n, i)\in \D_n\\
				x,  &  (n,i) \in \V_n^\M \\
				0 & (n, i)\in \U_n. \end{array}\right.
		\end{align*}
		Then $f_{\bullet}: {X}_{\bullet}^\M   \rightarrow X_{\bullet}$ and $ g_{\bullet}: X_{\bullet}   \rightarrow {X}_{\bullet}^\M $ are chain maps which are homotopy equivalent:		
$gf=\Id_{{X}_{\bullet}^\M}$ and $fg\sim \Id_{X_{\bullet}}$ via the homotopy
		\begin{align*}
			{\theta_n}: X_n & \rightarrow X_{n+1} \\
			x \in X_{n,i} & \mapsto {\theta_n}(x):= \left\{\begin{array}{ll} \sum\limits_{(n+1, j)\in \U_{n+1}}\sum\limits_{p\in \Path^\M((n,i), (n+1, j)) } \varphi^\M_p(x),& (n, i)\in \D_n\\
				0 & otherwise. \end{array}\right.
		\end{align*}
	\end{itemize}
\end{thm}

Next, we consider the reduced two-sided bar resolution $B(A,A)$ of $A$ (see Appendix~\ref{Sect: TTCalculus}), which has  the following direct sum decomposition: 
$$B(A,A)_n = A \otimes  \overline{A}^{\otimes  n} \otimes  A \cong \bigoplus_{w_1,\cdots, w_n \in \B_+} A \otimes  \bfk \{(w_1,\cdots, w_n) \} \otimes  A , \quad \text{for}~ n\geqslant 1, $$
where $\B_{+} = \B \setminus \{ 1 \}$ is defined in Corollary~\ref{cor:k-basis_A}.
The Morse matching presented below is derived from the work of Sk\"{o}ldberg in \cite{S06}.
\begin{prop} [\cite{S06}]\label{prop: Morse matching for Anick resolution} Let  $\M$ be a full subquiver of $Q_{B}$ consisting  of the following arrows
	\begin{itemize}
		 \item[(i)] $  (x,u_1,\cdots,u_n) \xrightarrow{-1\otimes  1} (xu_1,\cdots,u_n);$
		 \item[(ii)] $ (y,v_1,\cdots,v_n) \xrightarrow{-1\otimes  1} (yv_1,\cdots,v_n);$
		 \item[(iii)] $ (z,w_1,\cdots,w_n) \xrightarrow{-1\otimes  1} (zw_1,\cdots,w_n);$
		 \item[(iv)] $ (z,y, v_1, \cdots, v_n)   \xrightarrow{1\otimes  1} (z, yv_1,\cdots,v_n);$
		 \item[(v)] $\begin{tikzcd}[baseline = 15.5pt]
		                 (x,y^{k_1}x, \cdots, y^{k_n}x ,u_1', \cdots,u_m') \ar[rd, "(-1)^{n+1}\otimes  1"]& \\
 & (x,y^{k_1}x, \cdots, y^{k_n}x u_1',\cdots,u_m');
		             \end{tikzcd} $
		 \item[(vi)]   $\begin{tikzcd}[baseline = 15.5pt]
   (x,y^{k_1}x, \cdots, y^{k_{n-1}}x,z ,v_1',\cdots,v_m')\ar[rd, "(-1)^{n+1}"] & \\
&  \hspace{-1cm}(x,y^{k_1}x, \cdots, y^{k_{n-1}}x,z v_1',\cdots,v_m'); \end{tikzcd} $
		 \item[(vii)] $\begin{tikzcd}[baseline = 15.5pt]
		           (x,y^{k_1}x, \cdots, y^{k_{n-1}}x,z, y,w_1',\cdots,w_m')\ar[rd, "(-1)^{n+2}\otimes  1"] & \\
&  \hspace{-1.4cm}(x,y^{k_1}x, \cdots, y^{k_{n-1}}x,z,y w_1',\cdots,w_m'),
		               \end{tikzcd}$
	\end{itemize}
	for $m,n\geqslant 1, k_1,\cdots, k_n \geqslant 0,$ and $xu_1, yv_1, zw_1, y^{k_n}x u_1', z v_1', y w_1' \in \mathcal{B}_+.$
Then $\mathcal{M}$ satisfies the condition in Proposition~\ref{Prop: sufficient condition for Morse matching} and is therefore a Morse matching for which the Morse complex $B(A,A)^{\mathcal{M}}$ is the two-sided Anick resolution of $A$.
\end{prop}

We now extend the Morse matching defined by Sk\"{o}ldberg by adding more arrows so that it yields the two-sided Koszul  resolution    of $A$.

\begin{thm}\label{prop: bigger morse matching for bar}
Define the new full subquiver $\widetilde{\M}$ of $Q_B$  as $\widetilde{\M} = \M \cup \M',$ where
$$\M' :=  \left\{   \begin{array}{l} (x,x,y^{\ell_1}x,\cdots , y^{\ell_n} x) \xrightarrow{-1\otimes  1}  (x,y^{\ell_1+1}x,\cdots , y^{\ell_n} x), \\
 (x,x,y^{\ell_1}x,\cdots , y^{\ell_n} x,z) \xrightarrow{-1\otimes  1}  (x,y^{\ell_1+1}x,\cdots , y^{\ell_n} x,z),\\
   (x,x,y^{\ell_1}x,\cdots , y^{\ell_n} x,z,y) \xrightarrow{-1\otimes  1}  (x,y^{\ell_1+1}x,\cdots , y^{\ell_n} x,z,y),  \end{array}     \right\}_{\substack{n\geqslant 1,\\ \ell_1,\cdots, \ell_n\geqslant 0.}} $$
 Then $\widetilde{\M}$  is a Morse matching, and the resulting Morse complex is the two-sided Koszul resolution of $A$.
	\end{thm}

\begin{proof}
Note that $Q^{\widetilde{\M}}$ is obtained from $Q^{\M}$ by replacing the arrows in $\M'$  with their reverse dotted arrows.  Let $p$ be a zigzag path in $Q^{\widetilde{\M}}.$
If  $p$ does not contain any dotted arrows corresponding to $\M'$, then $p$ is a path in $Q^{\mathcal{M}}$, and hence by Proposition~\ref{prop: Morse matching for Anick resolution}, it is of finite length.
Now, suppose that $p$ contains a dotted arrow induced by $\M'$. Without loss of generality, let the first such arrow encountered along the path $p$  be
$$ (x,y^{\ell_1+1}x,\cdots , y^{\ell_n} x)  \dashrightarrow  (x,x, y^{\ell_1}x,\cdots , y^{\ell_n} x);$$
the other two types of arrows can be discussed similarly.
A direct computation shows that, if it exists, the next dashed arrow induced by $\M'$ after this one is of the form
$$(x,y^{\ell_1'+1}x,y^{\ell_2}x, \cdots , y^{\ell_n} x)  \dashrightarrow  (x,x, y^{\ell_1'}x,y^{\ell_2}x,  \cdots , y^{\ell_n} x), $$
where $\ell_1'< \ell_1.$
Therefore, only finitely many dashed arrows in $p$ are induced by $\M'$, which split $p$ into finitely many paths in $Q^\M$.
Consequently,  by Proposition~\ref{prop: Morse matching for Anick resolution},  $p$ has finite length.
Hence $\widetilde{\M}$ is a Morse matching according to Proposition~\ref{Prop: sufficient condition for Morse matching}.

Finally, we obtain the critical set of $\widetilde{\M}$ as
$$\V_4^{\widetilde{\M}} = \{ (x,x,z,y)\}, \quad \V_3^{\widetilde{\M}} = \{ (x,x,z), (x,z,y) \}, \quad   \V_2^{\widetilde{\M}} = \{(x,x), (x,z), (z,y)\}, $$
$$ \V_1^{\widetilde{\M}} = \{x,y,z\}, \quad   \V_0^{\widetilde{\M}} = \{  \ast_{A\ot A}  \}, \quad  \V_{-1}^{\widetilde{\M}} = \{  \ast_{A}  \}.$$
Under  the identification,
\begin{equation}\label{eq: equi of Koszul and minimal res}
	\begin{array}{c}	(x,x,z,y)   \leftrightarrow   (x+y)xzy,  \quad
	(x,x,z)  \leftrightarrow   (x+y)xz,   \quad
	(x,z,y)  \leftrightarrow    xzy,\\
	(x,x) \leftrightarrow (x+y)x, \quad
	(x,z) \leftrightarrow xz, \quad
	(z,y) \leftrightarrow zy,
	\end{array}  \end{equation}
a direct computation shows that   the Morse complex of $\widetilde{\M}$   is precisely the two-sided Koszul resolution   of $A$.
\end{proof}

\begin{remark}
	 In \cite{CLZ24}, the authors obtained the Kosuzl resolution of $A$ by constructing a Morse matching $\M''$  on the quiver  associated to the Anick resolution.
Our approach combines the Morse matchings of Sköldberg and \cite{CLZ24}, thereby achieving computational efficiency and requiring only a single application of algebraic Morse theory.
\end{remark}

We now state the explicit formula for the comparison morphisms  between the two-sided Koszul and two-sided bar resolutions of $A$. 
The detailed computation, which involves enumerating all zigzag paths to or from critical vertices, is lengthy and has been moved to Appendix~\ref{Sect: proof} for brevity.

\begin{thm}\label{thm:Comparision from K to B}
	The comparison morphism
	\begin{equation*}
 \iota : K \longrightarrow B(A,A)
	\end{equation*}
of $A^{e}$-module morphisms, from the  two-sided Koszul resolution  $K$ of $A$ to the two-sided  bar resolution $B(A,A)$ of $A$ is the canonical inclusion map. 
	\end{thm}

	\begin{thm}\label{thm:Comparision from B to K}
		The comparison morphism
		\begin{equation*}
			\pi:  B(A,A) \longrightarrow K
		\end{equation*}
		of $A^{e}$-module morphisms, from  the two-sided  bar resolution $B(A,A)$ of $A$ to  the  two-sided Koszul resolution   $K$ of $A$  is defined as follows.
	\begin{itemize}[leftmargin=0.9cm]
		\item[(i)] In degrees $-1$ and $0,$  $ \pi$ is the identity map$;$
		
		\item[(ii)] in degree $1$, we have

	\noindent	$\pi_1(1\otimes  a_1\cdots a_n \otimes  1 ) = \sum_{i=1}^{n} a_{1,i-1}\otimes  a_i \otimes  a_{i+1,n},$ with $ a_1,\cdots, a_n \in \{x,y,z\}; $
		
		\item[(iii)] in degree $2$, we have

 \noindent	$\pi_2(1\otimes  ax \otimes  zb \otimes  1) = a \otimes  xz \otimes  b,$
		
 \noindent		$\pi_2(1\otimes  a'z \otimes  yb' \otimes  1) = a' \otimes  zy  \otimes  b',$
		
 \noindent	$\pi_2(1\otimes  y^{k_0}xy^{j_1} \otimes  y^{k_1-j_1}xb'' \otimes  1) =y^{k_0}(x+y)^{k_1} \otimes (x+y)x \otimes  b''  $

 \noindent $\hspace{2cm}  - \sum_{i=0}^{k_1-1} y^{k_0}(x+y)^{i}  \otimes  (x+y)x \otimes  y^{k_1-i-1}xb'',$
		
 \noindent	$\pi_2( 1\otimes  a'zxy^{j_1} \otimes  y^{k_1-j_1}xb'' \otimes  1 )  =  -a'\otimes  zy  \otimes  y^{k_1}xb'' + a'zxy^{k_1-1} \otimes     (x+y)x \otimes  b'' $

 \noindent  $\hspace{2cm}  - \sum_{i=0}^{k_1-1} a'zxy^{i-1} \otimes   (x+y)x    \otimes  y^{k_1-i-1}xb'', $
		
\noindent		with $k_0\geqslant 0,0\leqslant j_1 \leqslant k_1,$ and $ a,a',b,b',b'' \in \B$ such that $ax,a'z, zb, yb',xb''\in \B;$
		
		\item[(iv)]  in degree $3$, we have
		
	\noindent	$\pi_3(1\otimes  axy^{j_1} \otimes  y^{k_1-j_1}x \otimes  zb \otimes  1) = a(x+y)^{k_1} \otimes  (x+y)xz \otimes  b,$
		
	\noindent	$\pi_3(1\otimes  ax \otimes  z \otimes  yb' \otimes  1) = a \otimes  xzy \otimes  b',$
		
	\noindent	$\pi_3(1\otimes  ax \otimes  zxy^{j_1} \otimes  y^{k_1-j_1}xb''\otimes  1) = -a\otimes  xzy \otimes  y^{k_1}xb'',$
		
	\noindent	with $0\leqslant j_1 \leqslant k_1,$ and $a,b,b',b'' \in \B $ such that $ax, zb,yb',xb'' \in \B;$
		
		\item[(v)] in degree $4$, we have
		
	\noindent	$\pi_4(1 \otimes  axy^{j_1} \otimes  y^{k_1-j_1}x \otimes  z \otimes  yb \otimes  1) = a(x+y)^{k_1}\otimes  (x+y)xzy \otimes  b,$
		
	\noindent	$\pi_4(1\otimes  axy^{j_1} \otimes  y^{k_1-j_1}x \otimes  zxy^{j_2} \otimes  y^{k_2-j_2}xb' \otimes  1)  = - a(x+y)^{k_1} \otimes  (x+y)xzy \otimes  y^{k_2}xb',$
		
	\noindent	with $0\leqslant j_1 \leqslant k_1, 0\leqslant j_2 \leqslant k_2,$and $a,b,b' \in \B$ such that $ax,yb,xb' \in \B;$

		\item[(vi)] on all remaining direct summands of $B(A,A)$, the morphism $\pi$ vanishes.
	\end{itemize}

\end{thm}

\subsection{Connes' Differential}\label{Sect: Connes' differential}

Based on the computations presented in Section \ref{Sect: Hochschild homology}, $\HH_{n}(A)=0$ for $n\geqslant 2,$ hence the only   non-trivial component of the Connes’ differential on Hochschild homology of $A$ is $B_0:\HH_0(A)\to \HH_1(A)$.

\begin{prop}
The Connes’ differential   $B_{\bullet}:\HH_{\bullet}(A)  \to  \HH_{\bullet + 1}(A) $  is determined by
   \begin{itemize}
   	\item[(i)] $B_0(\alpha_0(n)) = n \alpha_1(n), B_0(\beta_0(n)) = n \beta_1 (n),  B_0(\gamma_0(n)) = n \gamma_1(n),$
   	\item[(ii)] $B_0(\overline \epsilon_0 (n_{1,p};m_{1,p})) = \theta_1(n_{1,p};m_{1,p}),$
   	\item[(iii)] $B_0(\zeta_0) = 0,$
   \end{itemize}
    where $p, n,n_1,\cdots, n_p, m_1,\cdots, m_p  \geqslant  1$ are positive integers.
\end{prop}

\begin{proof}
According to the comparison morphisms between the Koszul resolution and the Bar resolution established in Section~\ref{Sect: Comparison morphisms}, the action of Connes' differential, denoted as $\overline{B_0}$, on the Koszul  $0$-chain   is given by:
$$ \overline{B_0}(a)  =  \pi_{\ast}(B_0(a))    =\sum_{i=1}^{n} a_{i+1}\cdots a_n a_1\cdots a_{i-1} \ot a_i, $$
where $a=a_1\cdots a_n\in \B, a_1,\cdots, a_n \in \{x,y,z\}.$
Then we have
\begin{align*}
	 \overline{B_0} (\alpha_0(n)) = & (-1)^{n-1} \overline{B_0} (y^{n-1}x) & \\
	  = & (-1)^{n-1} \big( y^{n-1} \ot x + (\sum_{i=1}^{n-1}y^{n-i-1}xy^{i-1}) \ot  y \big) & + \im (d_2) \\
	 \overset{(\ast)}{=}& (-1)^{n-2} n y^{n-2}x \ot x &  + \im (d_2) \\
	  = & n x^{n-1} \ot x = \alpha_1(n), &
\end{align*}
where the  equation $(\ast)$  holds   because the difference between the two sides equals
$d_2 \big( ( \sum_{j=1}^{n-2}y^{j-1}xy^{n-j-2}-y^{n -2} ) \ot (x+y)x \big) \in \im(d_2). $
The computation of the action of $B_0$ is on the remaining elements is straightforward.
\end{proof}

\subsection{Gerstenhaber Bracket} \label{Sect: Gerstenhaber bracket}
According to the comparison morphisms $\iota$ and $\pi$ computed in Section~\ref{Sect: Comparison morphisms}, the action of the Gerstenhaber bracket on Koszul cochains $f$ and $g$ is given by:
\begin{equation}\label{eq: Gbracket on K} [f,g] :  = \iota^{*} [\pi^{*}(f), \pi^{*}(g)], \end{equation}
where $ \iota^{*}  $ and $\pi^{*}$ are the pre-composition maps  induced by $\iota$ and $\pi,$ respectively.

Due to the graded-anticommutativity of the Gerstenhaber bracket, it is determined by its restriction to $\HH^{-n}(A) \otimes \HH^{-m}(A)$ for $n \geqslant m$.
Furthermore, $\HH^{-n}(A)$ is nonzero only for $n = 0, 1, 2, 4$, and $\HH^{0}(A) = k$.
Hence, the Gerstenhaber bracket is non-trivial only on
\begin{itemize}
    \item[(i)] $\HH^{-1}(A) \otimes \HH^{-1}(A) \to \HH^{-1}(A);$
    \item[(ii)] $\HH^{-2}(A) \otimes \HH^{-1}(A) \to \HH^{-2}(A);$
    \item[(iii)] $\HH^{-4}(A) \otimes \HH^{-1}(A) \to \HH^{-4}(A).$
\end{itemize}

For better readability, the explicit formula for the Gerstenhaber bracket on the Hochschild cohomology of $A$  has been deferred to Appendix~\ref{Sect: Result GBracket} because of its complexity.
Instead, an illustrative example will be provided at the end of this section to demonstrate the calculation process.

For $B^{-1}(n) = \left( \begin{array}{rcl}  x  & \mapsto &  y^{n} x \\  y &  \mapsto & y^{n+1}    \end{array} \right),$ $B^{-1}(m) = \left( \begin{array}{rcl}  x  & \mapsto &  y^{m} x \\  y &  \mapsto & y^{m+1}   \end{array} \right),$ we have

\noindent $[B^{-1}(n), B^{-1}(m)] = (m-n)B^{-1}(m+n)$ according to Example~\ref{Ex: [B,B]}.
Here $B^{-1}(n)$ behaves like the one-dimensional vector field 
$$ x^{n+1}\frac{\partial}{\partial x} = : X_n $$
with $ [X_n, X_m] = (m-n) X_{m+n}.  $
Furthermore,   $(\HH^{\bullet}(A), [-,-])$ contains the following subalgebras by virtue of Proposition~\ref{prop: gerstenhaber on -1 x -1 to -1} and \ref{prop: gerstenhaber on -4 -1 to -1}.
\begin{cor}\label{cor: sub Witt of HH}
Let $L_0 =  A^{-1}, L_{n} = B^{-1}(n),$  $L_{n}' = H^{-1}(;n),$ and $I_n =A^{-4}(n)$, for $n\geqslant 1.$
Then, under the Gerstenhaber bracket,
\begin{itemize}
\item[(i)] $ \{L_n\}_{n\geqslant 0}  $ is the non-negative part  of a Witt algebra;
\item[(ii)] $\{L_n'\}_{n\geqslant 1} $ is the positive part of a Witt algebra;
\item[(iii)] $\{ L_n, I_n\}_{n\geqslant 1}$ is the positive part of  semidirect products of   Witt algebra and   tensor density module $W(0,0).$
\end{itemize}
\end{cor}

Finally, as an illustration, the bracket $[B^{-1}(n), B^{-1}(m)]$ is calculated as follows.

\begin{exam}\label{Ex: [B,B]}
	 Let $f= B^{-1}(n)$ and  $g = B^{-1}(m)$ be the basis elements of $\HH^{-1}(A)$  given in Proposition~\ref{prop: HH^1}.
According to equation~\eqref{eq: Gbracket on K}, we have
$$\begin{array}{rl}
  {[f,g]}(x)  = &   \sum\limits_{i=0}^{m-1} y^{i}f (y)y^{m-i-1}x + y^{m}f(x) -\sum\limits_{i=0}^{n-1} y^{i}g(y)y^{n-i-1}x - y^n g(x)\\
   =  &   (m-n)y^{m+n}x, \\
  {[f,g]}(y)  = & \sum\limits_{i=0}^{m}y^{i} f(y) y^{m-i} -  \sum\limits_{i=0}^{n} y^{i}g (y)y^{n-i}\\
     = & (m-n)y^{m+n+1}, \\
  {[f,g]}(z)  =   &   0.
  \end{array}$$
Therefore,  we have
 $$[B^{-1}(n), B^{-1}(m)] =(m-n)B^{-1}(m+n). $$
\end{exam}

\appendix

\section{Results of Cap Product and Gerstenhaber Bracket}\label{Sect: Result GBracket}
In this appendix, $\delta_{m,n}$ denotes the Kronecker delta.
For notational convenience in the subsequent  formulas, we extend the basis indices of $\HH_{\bullet}(A)$ and $\HH^{\bullet}(A)$ (given in Section~\ref{Sect: Hochschild co + homo}) to $\mathbb{Z}$, setting all undefined terms to zero.

\begin{prop} The only nontrivial component of the cap product
	$$\cap : \HH_1(A) \otimes \HH^{-1}(A) \to \HH_{0}(A) $$    is given by the following  formulas.
	
	\begin{small}
\begin{itemize}[leftmargin = 2em]
\item  $ \alpha_1(n) \cap A^{-1} = - \alpha_0(n),$
\item[]  $ \alpha_1(n) \cap B^{-1}(m)= (-1)^{m+1} \alpha_0(m+n),$
\item[] $   \alpha_1(n) \cap C^{-1}(m)= (-1)^{m+1} \alpha_0(m+n), $
\item[] $ \alpha_1(n) \cap D^{-1}(m,j) = (-1)^{m+1} \alpha_0(m+n),$
 \end{itemize}

\begin{itemize}[leftmargin = 2em]
\item $\beta_1(n)\cap A^{-1}   = -\beta_0 (n),$
\item[] $\beta_1(n)\cap B^{-1}(m)  = - \beta_0 (m+n),$
\item[] $\beta_1(n)\cap C^{-1}(m)  = (-1)^{m+n+1}\alpha_0 (m+n),$
\item[] $  \beta_1(n)\cap D^{-1}(m,j) = (-1)^{m+n+1}\alpha_0 (m+n),$
\item[]  $\beta_1(n) \cap F^{-1}(r_{0,q+1}; s_{1,q}) =  - \overline \epsilon_0 (r_1+r_{q+1} + n -2, r_{2,q}; s_{1,q}),$
\end{itemize}

\begin{itemize}[leftmargin = 2em]
\item $\gamma_1(n) \cap C^{-1}(m) =  -\overline \epsilon_0 (m-1;n),$
\item[] $\gamma_1(n) \cap  G^{-1}(r_{1,q+1}; s_{1,q}) = \overline \epsilon_0 (r_{1,q}, r_{q+1}-1; s_{1,q},n),  $
\item[] $\gamma_1(n) \cap   H^{-1}(r_{2,q}; s_{1,q}) =  - \overline \epsilon_0(r_{2,q};s_{2,q-1},  s_1+s_q + n -1),$
\end{itemize}

\begin{itemize}[leftmargin = 2em]
\item $\theta_1(n_{1,p}; m_{1,p}) \cap A^{-1}=  -(\sum_{i=1}^{p} n_i +p) \overline \epsilon_0(n_{1,p}; m_{1,p}), $
\item[] $\theta_1(n_{1,p}; m_{1,p}) \cap B^{-1}(m) = -   \sum_{i=1}^{p} n_i   \overline \epsilon_0(n_{1,i-1}, n_i+m, n_{i+1,p}; m_{1,p}),$
\item[]  $\theta_1(n_{1,p}; m_{1,p}) \cap  C^{-1}(m) $
			
		\noindent  $ = - \sum_{i=1}^{p}\sum_{j=1}^{m_i-1}  \overline \epsilon_0 (n_{1,i}, m-1,n_{i+1,p}; m_{1,i-1}, j, m_i - j, m_{i+1, p}), $

\item[] $\theta_1(n_{1,p}; m_{1,p}) \cap F^{-1}(r_{1,q+1}; s_{1,q}) $
			
	\noindent	  $= -\sum_{i=1}^{p} \sum_{j=1}^{n_i} \overline \epsilon_0 (n_{1,i-1}, r_1+j-1, r_{2,q}, r_{q+1}+n_i - j-1 , n_{i+1, p}; m_{1,i-1},s_{1,q}, m_{i,p}) ,$

\item[] $\theta_1(n_{1,p}; m_{1,p}) \cap	G^{-1}(r_{1,q+1}; s_{1,q})$
			
	\noindent	  $ = \sum_{i=1}^{p} \sum_{j=1}^{m_i-1} \overline \epsilon_0 (n_{1,i},r_{1,q}, r_{q+1}-1,n_{i+1,p}; m_{1,i-1},j, s_{1,q}, m_i - j, m_{i+1,p}),	  $

\item[] $\theta_1(n_{1,p}; m_{1,p}) \cap H^{-1}(r_{2,q}; s_{1,q}) $
			
	\noindent	  $ = - \sum_{i=1}^{p} \sum_{j=1}^{m_i} \overline \epsilon_0 (n_{1,i}, r_{2,q}, n_{i+1,p}; m_{1,i-1}, s_1+j-1, s_{2,q-1}, s_{q} + m_i - j, m_{i+1,p}),$

\end{itemize}

\begin{itemize}[leftmargin = 2em]
\item  the cap product is zero on the remaining basis elements.		
	\end{itemize}

	\end{small}
\end{prop}

\begin{prop}\label{prop: gerstenhaber on -1 x -1 to -1}
The Gerstenhaber bracket on  the component
$$[-,-]: \HH^{-1}(A) \otimes \HH^{-1}(A)\to \HH^{-1}(A)$$
 is  determined by the following relations. Owing to antisymmetry $[X, Y] = -[Y, X]$, we list only the brackets $[X, Y]$ for basis elements $X$ preceding $Y$ in the ordering given in Proposition~\ref{prop: HH^1}.
\end{prop}
{ \small
\begin{itemize}[leftmargin = 2 em]
\item $[A^{-1}, A^{-1}] = 0,$
\item[ ] $[A^{-1}, B^{-1}(m)] = m B^{-1}(m),$
\item[ ] $[A^{-1}, C^{-1}(m)] = mC^{-1}(m),$
\item[ ] $[A^{-1}, D^{-1}(m,j)] = m D^{-1}(m,j),$
\item[ ]  $[A^{-1}, E^{-1}(r_{0, q+1}; s_{1, q})] =  (\sum_{i=0}^{q+1} r_i + q - 1 ) E^{-1}(r_{0, q+1}; s_{1, q}),$
\item[ ] $[A^{-1}, F^{-1}(r_{1, q+1}; s_{1, q})] =  (\sum_{i=1}^{q+1} r_i + q - 2 ) F^{-1}(r_{1, q+1}; s_{1, q}), $
\item[ ] $[A^{-1}, G^{-1}(r_{1, q+1}; s_{1, q})] = (\sum_{i=1}^{q+1} r_i+ q ) G^{-1}(r_{1, q+1}; s_{1, q}),$
\item[ ] $[A^{-1}, H^{-1}(r_{2, q}; s_{1, q})] = (\sum_{i=2}^{q}r_i+q-1)H^{-1}(r_{2, q}; s_{1, q}),$

\item  $  [B^{-1}(n), B^{-1}(m)] = (m-n) B^{-1}(m+n),$
\item[ ] $  [B^{-1}(n), C^{-1}(m)]=  (m-1)C^{-1}(m+n) - \sum_{i=1}^{n-1}D^{-1}(m+n,i), $

\item[ ]  $  [B^{-1}(n), D^{-1}(m,j)] =  j D^{-1}(m+n,n+j) + (m-j-1)D^{-1}(m+n,j) $ 

$- \sum_{s=j+1}^{n+j-1} D^{-1}(m+n,s), $
 		
 \item[ ]  	$  [B^{-1}(n),  E^{-1}(r_{0,q+1};s_{1,q})] = \sum_{i=0}^{q+1} r_i E^{-1}(r_{0,i-1},r_i+n, r_{i+1,q+1}; s_{1,q}) $

 $- \sum_{i=0}^{n-1} E^{-1}(r_0+i, r_{1,q}, r_{q+1} +n - i; s_{1,q}) - E^{-1}(r_{0,q}, r_{q+1}+n; s_{1,q}),   $
 \item[ ] $  [B^{-1}(n),  F^{-1}(r_{1,q+1}; s_{1,q}) ] =  \sum_{i=1}^{q+1} r_i F^{-1}(r_{1,i-1}, r_{i}+n, r_{i+1,q+1}; s_{1,q}) $

 $-\sum_{i=0}^{n} F^{-1}(r_1+i, r_{2,q},r_{q+1}+n-i; s_{1,q}) -F^{-1}(r_{1,q}, r_{q+1}+n; s_{1,q}),  $
 \item[ ] $  [B^{-1}(n),  G^{-1}(r_{1,q+1}; s_{1,q}) ] = \sum_{i=1}^{q+1} r_i G^{-1}(r_{1,i-1}, r_{i}+n, r_{i+1,q+1}; s_{1,q})$

  $- G^{-1}(r_{1,q}, r_{q+1}+n; s_{1,q}) -\sum_{i=1}^{n-1} E^{-1}(i,r_{1,q}, r_{q+1}+n-i+1; s_{1,q}),   $
  \item[ ]  $  [B^{-1}(n),  H^{-1}(r_{2,q}; s_{1,q}) ] =\sum_{i=2}^{q} r_i H^{-1}(r_{2,i-1},r_i + n, r_{i+1,q}; s_{1,q}),  $

  \item $  [C^{-1}(n), C^{-1}(m)] = \sum_{i=1}^{m+1} D^{-1}(m+n,i) - \sum_{i=1}^{n-1}D^{-1}(m+n,i), $
  \item[ ] $ [C^{-1}(n), D^{-1}(m,j)] =   j D^{-1}(m+n, n+j) + \sum_{s=j+1}^{m-1} D^{-1}(m+n, s)$ 
  
  $ - \sum_{s=j+1}^{n+j-1} D^{-1}(m+n,s), $
  
  \item[ ] $      [C^{-1}(n),  E^{-1}(r_{0,q+1};s_{1,q})]  = r_0 E^{-1}(r_0+n,r_{1,q+1};s_{1,q})$

  $+   \sum_{i=1}^{r_1} E^{-1}(r_0+i, n+r_1-i, r_{2,q+1}; s_{1,q})$

  $ -\sum_{i=1}^{n-1} E^{-1}(r_0+i, r_{1,q}, r_{q+1}+n-i; s_{1,q}) $

  $+ \sum_{i=1}^{q} \sum_{j=1}^{s_i-1} E^{-1}(r_{0,i}, n-1, r_{i+1,q+1}; s_{1,i-1}, j, s_i - j , s_{i+1,q}),  $

  \item[ ] $      [C^{-1}(n),  F^{-1}(r_{1,q+1};s_{1,q})] $

  $=    \sum_{i=1}^{r_1-1} E^{-1}(i, n+r_1-i-1, r_{2,q+1}; s_{1,q}) $

   $+ \sum_{i=0}^{n-2} G^{-1}(r_1+i, r_{2,q}, r_{q+1}+n-i-2; s_{1,q})  $

   $ +\sum_{i=1}^{q} \sum_{j=1}^{s_i-1} F^{-1}(r_{1,i},n-1, r_{i+1,q+1}; s_{1,i-1}, j, s_i-j, s_{i+1,q}), $

  \item[ ] $      [C^{-1}(n),  G^{-1}(r_{1,q+1};s_{1,q})]  $

  $  =   -\sum_{i=1}^{r_1} E^{-1}(i, n+r_1-i, r_{2,q}, r_{q+1}+1; s_{1,q})  $

  $- \sum_{i=1}^{n-1}E^{-1}(i, r_{1,q}, r_{q+1}+n-i+1; s_{1,q})   $

   $ + \sum_{i=1}^{q} \sum_{j=1}^{s_i-1} G^{-1}(r_{1,i},n-1, r_{i+1,q+1}; s_{1,i-1}, j, s_i -j, s_{i+1,q}),  $

\item[ ] $      [C^{-1}(n),  H^{-1}(r_{2,q};s_{1,q})] $

$ =  \sum_{i=1}^{q} \sum_{j=1}^{s_i-1} H^{-1}(r_{2,i}, n-1, r_{i+1,q}; s_{1,i-1}, j, s_i-j, s_{i+1,q}),  $

\item $ [D^{-1}(n,i), D^{-1}(m,j)] =j D^{-1}(m+n,n+j)- i D^{-1}(m+n, m+i) $

$ + \sum_{s=i+j+1}^{m+i-1} D^{-1}(m+n,s)- \sum_{s=i+j+1}^{n+j-1} D^{-1}(m+n,s)  ,$

\item[ ]  $ [D^{-1}(n,i),  E^{-1}(r_{0,q+1};s_{1,q})]   =r_0 E^{-1}(r_0+n,r_{1,q+1};s_{1,q})$

$ + \sum_{j=1}^{r_1} E^{-1}(r_0+i+j, n+r_1-i-j, r_{2,q+1}; s_{1,q}) $

$  - \sum_{j=i+1}^{n-1} E^{-1}(r_0+j,r_{1,q},r_{q+1}+n-j; s_{1,q}),      $

\item[ ]  $ [D^{-1}(n,i),  F^{-1}(r_{1,q+1};s_{1,q})] $

$  =  - \sum_{j=i}^{r_1+i-1} E^{-1}(j, r_1+n - j -1, r_{2,q+1}; s_{1,q})$

$  + \sum_{j=1}^{n-i} E^{-1}(i, r_1+j-1, r_{2,q}, r_{q+1} +n - i-j; s_{1,q}), $

  \item[ ] $      [D^{-1}(n,i),  G^{-1}(r_{1,q+1};s_{1,q})] $

  $=  \sum_{j=1}^{r_1} E^{-1}(i+j, n + r_1-i-j, r_{2,q}, r_{q+1}+1; s_{1,q})   $

   $- \sum_{j=i+1}^{n-1}  E^{-1}(j, r_{1,q}, r_{q+1}+ n-j+1; s_{1,q}),      $
\item[ ] $      [D^{-1}(n,i),  H^{-1}(r_{2,q};s_{1,q})] =0,$

\item  $  [E^{-1}(n_{0, p+1}; m_{1, p}), E^{-1}(r_{0, q+1};s_{1, q})] $

$ =  -  \sum_{i=1}^{r_1} E^{-1}(r_0 + n_0 +i, n_{1,p}, n_{p+1}+r_1-i-1, r_{2,q+1}; m_{1,p}, s_{1,q}) $

$+ \sum_{i=1}^{n_1} E^{-1}(n_0 + r_0+i, r_{1,q}, r_{q+1}+n_1-i-1, n_{2,p+1}; s_{1,q}, m_{1,p}),$
\item[ ]$ [E^{-1}(n_{0, p+1}; m_{1, p}),  F^{-1}(r_{1,q+1};s_{1,q})]$

$  =  \sum_{i=0}^{n_1-1} E^{-1}(n_0+i, n_{1,p}, n_{p+1}+r_1-i-2, r_{2,q+1}; m_{1,p}, s_{1,q})  $

$ -\sum_{i=0}^{n_{p+1}-2} E^{-1}(n_{0,p}, r_1+i, r_{2,q},,r_{q+1}+n_{p+1}-i-2; m_{1,p}, s_{1,q}),  $

$ -\sum_{i=1}^{p} \sum_{j=0}^{n_i-1} E^{-1}(n_{0,i-1}, r_1+j,r_{2,q}, r_{q+1}+n_i -j-2, n_{i+1,p+1}; m_{1,i-1}, s_{1,q}, m_{i,p}),  $

\item[ ] $      [E^{-1}(n_{0, p+1}; m_{1, p}),  G^{-1}(r_{1,q+1};s_{1,q})]    $

$ =  - \sum_{i=1}^{r_1} E^{-1}(n_0+i, n_{1,p}, n_{p+1}+r_1-i-1, r_{2,q}, r_{q+1}+1; m_{1,p}, s_{1,q})$

$ +\sum_{i=1}^{n_1} E^{-1}(n_0+i, r_{1,q}, r_{q+1}+n_1-i, n_{2,p+1}; s_{1,q}, m_{1,p})    $

$ +\sum_{i=1}^{p} \sum_{j=1}^{m_i-1} E^{-1}(n_{0,i}, r_{1,q}, r_{q+1}-1, n_{i+1,p+1}; m_{1,i-1},j, s_{1,q},m_i-j, m_{i+1,p}),   $

\item[ ] $      [E^{-1}(n_{0, p+1}; m_{1, p}), H^{-1}(;s_{1})] $

$ = - \sum_{i=1}^{p} m_i E^{-1}(n_{0,p+1}; m_{1,i-1}, s_1+m_i-1, m_{i+1,p}),  $

\item[ ] $[E^{-1}(n_{0, p+1}; m_{1, p}), H^{-1}(r_{2,q};s_{1,q})]_{\substack{q \geqslant 2}}  $

$ =-\sum_{i=1}^{p} \sum_{j=0}^{m_i-1} E^{-1}(n_{0,i}, r_{2,q},n_{i+1,p+1}; m_{1,i-1}, s_1+j, s_{2,q-1}, s_{q}+m_i-j-1, m_{i+1,p}),   $

\item   $ [F^{-1}(n_{1, p+1}; m_{1, p}),  F^{-1}(r_{1,q+1};s_{1,q})]$

$ = \sum_{i=1}^{q}\sum_{j=0}^{r_i-1} F^{-1}(r_{1,i-1}, n_1+j, n_{2,p}, n_{p+1}+r_i-j-2, r_{i+1,q+1}; s_{1,i-1}, m_{1,p}, s_{i,q})   $

$ + \sum_{i=0}^{r_{q+1}-2} F^{-1}(r_{1,q}, n_{1}+i, n_{2,p}, n_{p+1}+r_{q+1}-i-2; s_{1,q}, m_{1,p})   $

$ -\sum_{i=0}^{n_{p+1}-2} F^{-1}(n_{1,p}, r_1+j, r_{2,q}, r_{q+1}+n_{p+1}-j-2; m_{1,p}, s_{1,q}),    $

$ -\sum_{i=1}^{p} \sum_{j=0}^{n_i-1} F^{-1}(n_{1,i-1}, r_{1}+j, r_{2,q}, r_{q+1}+n_i-j-2, n_{i+1,p+1}; m_{1,i-1}, s_{1,q}, m_{i,p})   $

\item[ ] $ [F^{-1}(n_{1, p+1}; m_{1, p}),  G^{-1}(r_{1,q+1};s_{1,q})]$

$ = \sum_{i=1}^{q} \sum_{j=0}^{r_i-1} G^{-1}(r_{1,i-1}, n_1+j, n_{2,p}, n_{p+1}+r_i-j-2, r_{i+1,q+1}; s_{1,i-1}, m_{1,p}, s_{i,q})   $

$ +\sum_{i=0}^{r_{q+1}-2} G^{-1}(r_{1,q},n_1+i,n_{2,p}, n_{p+1}+r_{q+1}-i-2; s_{1,q}, m_{1,p})    $

$ -\sum_{i=1}^{n_1-1} E^{-1}(i, r_{1,q},r_{q+1}+n_1-i-1,n_{2,p+1}; s_{1,q}, m_{1,p})    $

$ + \sum_{i=1}^{p} \sum_{j=1}^{m_{i}-1} F^{-1} (n_{1,i}, r_{1,q}, r_{q+1}-1, n_{i+1,p+1}; m_{1,i-1},j, s_{1,q},m_i-j, m_{i+1,p})  $

\item[ ] $    [F^{-1}(n_{1, p+1}; m_{1, p}), H^{-1}(;s_{1})] $

$= - \sum_{i=1}^{p}m_i F^{-1}(n_{1,p+1}; m_{1,i-1}, s_{1}+m_i-1, m_{i+1,p}),    $

\item[ ] $    [F^{-1}(n_{1, p+1}; m_{1, p}), H^{-1}(r_{2,q};s_{1,q})]_{\substack{q \geqslant 2}}  $

$= - \sum_{i=1}^{p} \sum_{j=0}^{m_i-1} F^{-1}(n_{1,i}, r_{2,q}, n_{i+1,p+1}; m_{1,i-1}, s_{1}+j, s_{2,q-1}, s_{q}+m_i-j-1, m_{i+1,p})    $

$ + \sum_{i=2}^{q} \sum_{j=0}^{r_i-1} H^{-1}(r_{2,i-1}, n_1+j, n_{2,p}, n_{p+1}+r_i-j-2, r_{i+1,q}; s_{1,i-1}, m_{1,p}, s_{i,q}),   $

\item   $ [G^{-1}(n_{1, p+1}; m_{1, p}),  G^{-1}(r_{1,q+1};s_{1,q})]$

$ = - \sum_{i=1}^{r_1} E^{-1}(i,n_{1,p}, n_{p+1}+r_1-i, r_{2,q}, r_{q+1}+1; m_{1,p}, s_{1,q})   $

$ + \sum_{i=1}^{n_1} E^{-1}(i, r_{1,q}, r_{q+1}+n_1-i, n_{2,p}, n_{p+1}+1; s_{1,q}, m_{1,p})    $

$ -\sum_{i=1}^{q}\sum_{j=1}^{s_i-1}  G^{-1}(r_{1,i}, n_{1,p}, n_{p+1}-1, r_{i+1,q+1}; s_{1,i-1}, j, m_{1,p}, s_{i}-j, s_{i+1,q})  $

$ +\sum_{i=1}^{p} \sum_{j=1}^{m_i-1} G^{-1}(n_{1,i},r_{1,q}, r_{q+1}-1, n_{i+1,p+1}; m_{1,i-1}, j, s_{1,q}, m_i-j, m_{i+1,p}),    $

\item[ ] $    [G^{-1}(n_{1, p+1}; m_{1, p}), H^{-1}(;s_{1})]  $

$ = - \sum_{i=1}^{p} m_i G^{-1}(n_{1,p+1}; m_{1,i-1}, s_1+m_i-1, m_{i+1,p}) $

$ -\sum_{j=1}^{s_1-1} H^{-1}(n_{1,p}, n_{p+1}-1;  j, m_{1,p}, s_1-j),   $

\item[ ] $    [G^{-1}(n_{1, p+1}; m_{1, p}), H^{-1}(r_{2,q};s_{1,q})]_{\substack{q \geqslant 2}}   $

$ = - \sum_{i=1}^{p} \sum_{j=0}^{m_i-1} G^{-1}(n_{1,i}, r_{2,q}, n_{i+1,p+1}; m_{1,i-1}, s_1+j, s_{2,q-1}, s_q+m_i-j-1, m_{i+1,p}) $

$ -\sum_{i=1}^{q} \sum_{j=1}^{s_i-1} H^{-1}(r_{2,i}, n_{1,p}, n_{p+1}-1, r_{i+1,q}; s_{1,i-1}, j, m_{1,p}, s_i-j, s_{i+1,q}),   $

\item $[H^{-1}(;n), H^{-1}(; m)] = (m-n) H^{-1}(; m+n),$

\item[ ]   $[H^{-1}(n_{2, p}; m_{1, p}), H^{-1}(;s_{1})]_{\substack{p \geqslant 2}}   $

$ =\sum_{j=0}^{s_1-1} H^{-1}(n_{2,p};  m_{1}+j, m_{2,p-1}, m_{p}+ s_1-j-1)    $

$ -\sum_{i=1}^{p}m_i H^{-1}(n_{2,,p}; m_{1,i-1}, s_1+m_i-1, m_{i+1,p})   $

\item[ ]   $[H^{-1}(n_{2, p}; m_{1, p}), H^{-1}(r_{2,q};s_{1,q})]_{\substack{p,q \geqslant 2}}   $

$ =\sum_{i=1}^{q}\sum_{j=0}^{s_i-1} H^{-1}(r_{2,i}, n_{2,p}, r_{i+1,q}; s_{1,i-1}, m_{1}+j, m_{2,p-1}, m_{p}+ s_i-j-1, s_{i+1,q})    $

$ -\sum_{i=1}^{p}\sum_{j=0}^{m_i-1} H^{-1}(n_{2,i}, r_{2,q}, n_{i+1,p}; m_{1,i-1}, s_1+j, s_{2,q-1}, s_q+m_i-j-1, m_{i+1,p}).   $

\end{itemize}

}

\begin{prop}\label{prop: gerstenhaber on -2 -1 to -1}
The Gerstenhaber bracket on  the component
$$[-,-]: \HH^{-2}(A) \otimes \HH^{-1}(A)\to \HH^{-2}(A)$$
 is  determined by the following relations.
\end{prop}

 {\small
 \begin{itemize}[leftmargin = 2em]
 \item All brackets $[A^{-2}, f]$ with $f$ a basis element of $\HH^{-1}(A)$ are zero, except $[A^{-2}, A^{-1}] = A^{-2}.$
 \item The brackets  $[B^{-2}(n_{2, p};m_{1, p}), x]$ with $x \in \HH^{-1}(A)$ is given by:
 \item[ ] $ [B^{-2}(n_{2, p};m_{1, p}), A^{-1}] = - (\sum_{i=2}^{p}n_i + p -2) B^{-2}(n_{2, p};m_{1, p}), $
 \item[ ] $ [B^{-2}(n_{2, p};m_{1, p}), B^{-1}(m)] =  -\sum_{i=2}^{p}n_i B^{-2}(n_{2,i-1},n_i+m,n_{i+1,p}; m_{1,p}),  $
 \item[ ] $ [B^{-2}(n_{2, p};m_{1, p}), C^{-1}(m)]$

 $  = -\sum_{i=1}^{p}\sum_{j=1}^{m_i-1} B^{-2}(n_{2,i},m-1,n_{i+1,p}; m_{1,i-1},j, m_i-j, m_{i+1,p}), $

 \item[ ] $ [B^{-2}(n_{2, p};m_{1, p}), D^{-1}(m,j)] =0,$
 \item[ ] $ [B^{-2}(n_{2, p};m_{1, p}), E^{-1}(r_{0, q+1}; s_{1, q})] =0,$
 \item[ ]  $ [B^{-2}(n_{2, p};m_{1, p}), F^{-1}(r_{1, q+1}; s_{1, q})] $

 $ = - \sum_{i=2}^{p} \sum_{j=0}^{n_i-1} B^{-2}(n_{2,i-1}, r_1+j, r_{2,q}, r_{q+1}+n_i-j-2, n_{i+1,p}; m_{1,i-1}, s_{1,q}, m_{i,p}), $

 \item[ ]  $ [B^{-2}(n_{2, p};m_{1, p}), G^{-1}(r_{1, q+1}; s_{1, q})] $

 $ = \sum_{i=1}^{p} \sum_{j=1}^{m_i-1} B^{-2}(n_{2,i}, r_{1,q}, r_{q+1}-1, n_{i+1,p}; m_{1,i-1}, j, s_{1,q}, m_i-j, m_{i+1,p}),   $

 \item[ ] $ [B^{-2}(n_{2, p};m_{1, p}), H^{-1}(; s_{1})]  = -\sum_{i=1}^{p} m_i B^{-2}(n_{2,p}; m_{1,i-1},s_1+m_i-1,m_{i+1,p}),   $

 \item[ ] $ [B^{-2}(n_{2, p};m_{1, p}), H^{-1}(r_{2, q}; s_{1, q})]_{\substack{q \geqslant 2}}  $

 $ = -\sum_{i=1}^{p} \sum_{j=0}^{m_i-1} B^{-2}(n_{2,i}, r_{2,q}, n_{i+1,p}; m_{1,i-1},j+s_1, s_{2,q-1}, s_q+m_i-j-1,m_{i+1,p}).   $
 \end{itemize}

 }

\begin{prop}\label{prop: gerstenhaber on -4 -1 to -1}
The Gerstenhaber bracket on  the component
$$[-,-]: \HH^{-4}(A) \otimes \HH^{-1}(A)\to \HH^{-4}(A)$$
 is  determined by the following relations.
\end{prop}
{\small
\begin{itemize}[leftmargin=2em]
\item   $[A^{-4}, A^{-1}] = 3A^{-4},$
\item[ ] $[A^{-4}(n), A^{-1}] = -(n-2)A^{-4}(n),$
\item[ ] $ [B^{-4}(n_{0,p};m_{1,p}), A^{-1}] = - (\sum_{i=0}^{p}n_i + p -4) B^{-4}(n_{0,p}; m_{1,p}),  $
\item[ ] $ [C^{-4}(n_{0,p};m_{1,p}), A^{-1}] = - (\sum_{i=0}^{p}n_i + p -3) C^{-4}(n_{0,p}; m_{1,p}),  $
\item[ ] $ [D^{-4}(n_{2,p};m_{1,p}), A^{-1}] = - (\sum_{i=2}^{p}n_i + p -4) D^{-4}(n_{2,p}; m_{1,p}),   $
\item[ ] $ [E^{-4}(n_{2,p};m_{1,p}), A^{-1}] = - (\sum_{i=2}^{p}n_i + p -3) E^{-4}(n_{2,p}; m_{1,p}),   $

\item   $[A^{-4}, B^{-1}(m)] = 0,$
\item[ ] $[A^{-4}(n), B^{-1}(m)] = -nA^{-4}(m+n),$
\item[ ] $ [B^{-4}(n_{0,p};m_{1,p}), B^{-1}(m)] = -(n_0-1) B^{-4}(n_0+m, n_{1,p}; m_{1,p})$

$  - \sum_{i=1}^{p} n_i B^{-4}(n_{0,i-1}, n_i+m, n_{i+1,p}; m_{1,p}), $

\item[ ] $ [C^{-4}(n_{0,p};m_{1,p}), B^{-1}(m)] = -(n_0-1) C^{-4}(n_0+m, n_{1,p}; m_{1,p})$

$  - \sum_{i=1}^{p} n_i C^{-4}(n_{0,i-1}, n_i+m, n_{i+1,p}; m_{1,p}), $

\item[ ] $ [D^{-4}(n_{2,p};m_{1,p}), B^{-1}(m)] = -B^{-4}(1,m-1, n_{2,p}; m_{1,p})$

$ -\sum_{i=2}^{p}n_i D^{-4}(n_{2,i-1}, n_i+m, n_{i+1,p}; m_{1,p}),  $

\item[ ] $ [E^{-4}(n_{2,p};m_{1,p}), B^{-1}(m)] = -C^{-4}(1,m-1, n_{2,p}; m_{1,p})$

$ -\sum_{i=2}^{p}n_i E^{-4}(n_{2,i-1}, n_i+m, n_{i+1,p}; m_{1,p}),  $

\item   $[A^{-4}, C^{-1}(m)] =0, $

\item[ ]  $[A^{-4}(n), C^{-1}(m)] = -nB^{-4}(m+n), $

\item[ ] $ [B^{-4}(n_{0,p};m_{1,p}), C^{-1}(m)] = - (n_0-1)B^{-4}(n_0+m, n_{1,p}; m_{1,p}) $

$-\sum_{i=1}^{n_1} B^{-4}(n_0+i, n_1+m-i, n_{2,p}; m_{1,p}) $

$ - \sum_{i=1}^{p} \sum_{j=1}^{m_i-1} B^{-4}(n_{0,i}, m-1, n_{i+1,p}; m_{1,i-1}, j, m_i -j, m_{i+1,p}),  $

\item[ ] $ [C^{-4}(n_{0,p};m_{1,p}), C^{-1}(m)]  = -(n_0-1) C^{-4}(n_0+m,n_{1,p};m_{1,p})$

$ -\sum_{i=1}^{n_1} C^{-4}(n_0+i, n_1+m-i, n_{2,p}; m_{1,p}) $

$ - \sum_{i=1}^{p} \sum_{j=1}^{m_i-1} C^{-4}(n_{0,i}, m-1, n_{i+1,p}; m_{1,i-1}, j, m_i -j, m_{i+1,p}),  $

\item[ ] $ [D^{-4}(n_{2,p};m_{1,p}), C^{-1}(m)] =   -B^{-4}(1,m-1,n_{2,p}; m_{1,p})$

$  - \sum_{i=1}^{p}\sum_{j=1}^{m_i-1} D^{-4}(n_{2,i}, m-1, n_{i+1,p}; m_{1,i-1},j, m_i-j, m_{i+1,p}),  $

\item[ ] $ [E^{-4}(n_{2,p};m_{1,p}), C^{-1}(m)] =   -C^{-4}(1,m-1,n_{2,p}; m_{1,p})$

$  - \sum_{i=1}^{p}\sum_{j=1}^{m_i-1} E^{-4}(n_{2,i}, m-1, n_{i+1,p}; m_{1,i-1},j, m_i-j, m_{i+1,p}),  $

\item    $[A^{-4}, D^{-1}(m,i)] =0,$

\item[ ] $[A^{-4}(n), D^{-1}(m,i)] = -nA^{-4}(m+n),$

\item[ ] $ [B^{-4}(n_{0,p};m_{1,p}), D^{-1}(m,i)]  = -(n_0-1)B^{-4}(n_0+m,n_{1,p}; m_{1,p}) $

$ - \sum_{j=1}^{n_1} B^{-4}(n_0+i+j, n_1+m-i-j, n_{2,p}; m_{1,p}), $

\item[ ] $ [C^{-4}(n_{0,p};m_{1,p}), D^{-1}(m,i)]  =  -(n_0-1) C^{-4}(n_0+m,n_{1,p}; m_{1,p}) $

$ - \sum_{j=1}^{n_1} C^{-4}(n_0+i+j, n_1+m-i-j, n_{2,p}; m_{1,p}), $

\item[ ] $ [D^{-4}(n_{2,p};m_{1,p}), D^{-1}(m,i)] =  -B^{-4}(i+1, m-i-1, n_{2,p}; m_{1,p}),  $

\item[ ] $ [E^{-4}(n_{2,p};m_{1,p}), D^{-1}(m,i)] =  -C^{-4}(i+1, m-i-1, n_{2,p}; m_{1,p}),  $

\item   $[A^{-4}, E^{-1}(r_{0, q},1; s_{1, q})]  =
  B^{-4}(r_0+1, r_{1,q}; s_{1,q}) + C^{-4}(r_{0,q}; s_{1,q}),  $

\item[ ]
  $[A^{-4}, E^{-1}(r_{0, q},2; s_{1, q})] =   C^{-4}(r_0+1, r_{1,q}; s_{1,q}),   $

\item[ ]
 $ [A^{-4}, E^{-1}(r_{0, q},r_{q+1}; s_{1, q})]_{\substack{r_{q+1} \geqslant 3}}  =  0,
 $

\item[ ] $[A^{-4}(n), E^{-1}(r_{0, q+1}; s_{1, q})] = - \delta_{r_{q+1},1} C^{-4}(r_0+n+1, r_{1,q}; s_{1,q}), $

\item[ ] $[B^{-4}(n_{0,p};m_{1,p}), E^{-1}(r_{0, q+1}; s_{1, q})] =\delta_{r_0,1} \delta_{r_{q+1},1} C^{-4}(n_{0,p}, r_{1,q}; m_{1,p}, s_{1,q})   $

$ + \delta_{r_{q+1}, 1} \delta_{n_0,1} B^{-4}(r_0+1, r_{1,q}, n_{1,p}; s_{1,q}, m_{1,p})  $

$ + \sum_{i=1}^{n_1} B^{-4}(n_0+r_0+i, r_{1,q}, r_{q+1} + n_1-i-1, n_{2,p}; s_{1,q}, m_{1,p}),    $

\item[ ]  $[C^{-4}(n_{0,p};m_{1,p}), E^{-1}(r_{0, q+1}; s_{1, q})] =  \delta_{r_{q+1},1} \delta_{n_0,1} C^{-4}(r_0+1,r_{1,q}, n_{1,p}; s_{1,q}, m_{1,p})    $

$+  \sum_{i=1}^{n_1} C^{-4}(n_0+r_0+i, r_{1,q}, r_{q+1}+n_1-i-1, n_{2,p}, s_{1,q}, m_{1,p}),   $

\item[ ] $[D^{-4}(n_{2,p};m_{1,p}), E^{-1}(r_{0, q}, 1 ; s_{1, q})] =  \delta_{r_0,1} E^{-4}(n_{2,p}, r_{1,q}; m_{1,p}, s_{1,q}) $

$ + B^{-4}(r_0+1, r_{1,q}, n_{2,p}; s_{1,q-1}, s_{q}+m_1, m_{2,p}),$

\item[ ]  $[D^{-4}(n_{2,p};m_{1,p}), E^{-1}(r_{0, q}, r_{q+1}; s_{1, q})]_{\substack{r_{q+1}\geqslant 2}} =   B^{-4}(r_0+1, r_{1,q}, r_{q+1}-2, n_{2,p}; s_{1,q},m_{1,p}),$

\item[ ] $[E^{-4}(n_{2,p};m_{1,p}), E^{-1}(r_{0, q}, 1 ; s_{1, q})] = C^{-4}(r_0+1, r_{1,q}, n_{2,p}; s_{1,q-1}, s_{q}+m_1, m_{2,p}) $

\item[ ]  $[E^{-4}(n_{2,p};m_{1,p}), E^{-1}(r_{0, q}, r_{q+1}; s_{1, q})]_{\substack{r_{q+1}\geqslant 2}}=   C^{-4}(r_0+1, r_{1,q}, r_{q+1}-2, n_{2,p}; s_{1,q},m_{1,p}),$

\item    $[A^{-4}, F^{-1}(1,r_{2, q}, 1 ;s_{1, q})] = E^{-4}(r_{2,q}; s_{1,q}),$

\item[ ] $[A^{-4}, F^{-1}(r_1,r_{2, q}, 1 ;s_{1, q})]_{\substack{r_{1}\geqslant 2}} =   -B^{-4}(1,r_1-1, r_{2,q}; s_{1,q}) + C^{-4}(1,r_1-2, r_{2,q}; s_{1,q}),$

\item[ ] $[A^{-4}, F^{-1}(r_1,r_{2, q}, 2 ;s_{1, q})]_{\substack{r_{1}\geqslant 2}} =   -C^{-4}(1,r_1-1, r_{2,q}; s_{1,q}), $

\item[ ]  $[A^{-4}, F^{-1}(r_{1, q+1};s_{1, q})] =0, $ for  $F^{-1}(r_{1, q+1};s_{1, q})$ outside the above three cases,

\item[ ]  $[A^{-4}(n), F^{-1}(r_{1, q+1};s_{1, q})] = \delta_{r_{q+1},1} C^{-4}(n+1, r_1-1, r_{2,q}; s_{1,q}), $

\item[ ] $[B^{-4}(n_{0,p};m_{1,p}), F^{-1}(r_{1, q+1}; s_{1, q})] $

$=  - \delta_{r_{q+1},1} \delta_{n_0,1} B^{-4}(1,r_1-1, r_{2,q}, n_{1,p}; s_{1,q}, m_{1,p})  $

$+ \delta_{r_1,1} \delta_{r_{q+1},1} C^{-4}(n_{0,p}, r_{2,q}; m_{1,p-1}, m_p+s_1, s_{2,q})   $

$ -\sum_{i=1}^{p} \sum_{j=0}^{n_i-1} B^{-4}(n_{0,i-1}, r_1+j, r_{2,q}, r_{q+1}+n_i-j-2, n_{i+1,p}; m_{1,i-1}, s_{1,q}, m_{i,p}),  $

\item[ ] $[C^{-4}(n_{0,p};m_{1,p}), F^{-1}(r_{1, q+1}; s_{1, q})] $

$=  - \delta_{r_{q+1},1} \delta_{n_0,1} C^{-4}(1,r_1-1, r_{2,q}, n_{1,p}; s_{1,q}, m_{1,p})  $

$ +\delta_{r_{q+1},1} C^{-4}(n_{0,p},r_1-1, r_{2,q}; m_{1,p}, s_{1,q})   $

$ -\sum_{i=1}^{p} \sum_{j=0}^{n_i-1} C^{-4}(n_{0,i-1}, r_1+j, r_{2,q}, r_{q+1}+n_i-j-2, n_{i+1,p}; m_{1,i-1}, s_{1,q}, m_{i,p}),  $

\item[ ] $[D^{-4}(n_{2,p};m_{1,p}), F^{-1}(r_{1, q+1}; s_{1, q})] $

$ =   -\delta_{r_{q+1},1} B^{-4}(1,r_1-1, r_{2,q}, n_{2,p}; s_{1,q-1}, s_q+m_1, m_{2,p}) $

$ -B^{-4}(1, r_1-1, r_{2,q}, r_{q+1}-2, n_{2,p}; s_{1,q}, m_{1,p}) $

$+ \delta_{r_1,1} \delta_{r_{q+1},1} E^{-4}(n_{2,p}, r_{2,q}; m_{1,p-1}, m_p+s_1, s_{2,q})   $

$ - \sum_{i=2}^{p} \sum_{j=0}^{n_i-1} D^{-4}(n_{2,i-1}, r_1+j, r_{2,q}, r_{q+1}+n_i-j-2, n_{i+1,p}; m_{1,i-1}, s_{1,q}, m_{i,p}),   $

\item[ ] $[E^{-4}(n_{2,p};m_{1,p}), F^{-1}(r_{1, q+1}; s_{1, q})] $

$= \delta_{r_{q+1},1} E^{-4}(n_{2,p}, r_1-1, r_{2,q}; m_{1,p}, s_{1,q})$

$ - C^{-4}(1,r_1-1, r_{2,q}, r_{q+1}-2, n_{2,p}; s_{1,q}, m_{1,p})   $

$    -\delta_{r_{q+1},1} C^{-4}(1,r_1-1, r_{2,q}, n_{2,p}; s_{1,q-1}, s_q+m_1, m_{2,p})  $

$ -\sum_{i=2}^{p}\sum_{j=0}^{n_i-1} E^{-4}(n_{2,i-1}, r_1+j, r_{2,q}, r_{q+1}+n_i-j-2, n_{i+1,p}; m_{1,i-1}, s_{1,q}, m_{i,p}),  $

\item   $[A^{-4}, G^{-1}(r_{1, q+1};s_{1, q})] = \delta_{r_{q+1},1} C^{-4}(1,r_{1,q}; s_{1,q}),$

\item[ ] $[A^{-4}(n), G^{-1}(r_{1, q+1};s_{1, q})] = 0, $

\item[ ] $[B^{-4}(n_{0,p};m_{1,p}), G^{-1}(r_{1, q+1};s_{1, q})] =   \sum_{i=1}^{n_1} B^{-4}(n_0+i, r_{1,q}, r_{q+1}+n_1-i, n_{2,p}; s_{1,q}, m_{1,p}) $

$ + \sum_{i=1}^{p}\sum_{j=1}^{m_i-1} B^{-4}(n_{0,i}, r_{1,q}, r_{q+1}-1, n_{i+1,p}; m_{1,i-1},j, s_{1,q},m_i-j, m_{i+1,p})   $

$ + \delta_{r_{q+1},1} C^{-4}(n_{0,p}, r_{1,q}; m_{1,p}, s_{1,q}),  $

\item[ ] $[C^{-4}(n_{0,p};m_{1,p}), G^{-1}(r_{1, q+1};s_{1, q})] =   \sum_{i=1}^{n_1} C^{-4}(n_0+i, r_{1,q}, r_{q+1}+n_1-i, n_{2,p}; s_{1,q}, m_{1,p})   $

$ +\sum_{i=1}^{p}\sum_{j=1}^{m_i-1} C^{-4}(n_{0,i}, r_{1,q}, r_{q+1}-1, n_{i+1,p}; m_{1,i-1},j, s_{1,q},m_i-j, m_{i+1,p}),   $

\item[ ] $[D^{-4}(n_{2,p};m_{1,p}), G^{-1}(r_{1, q+1};s_{1, q})] =   B^{-4}(1,r_{1,q}, r_{q+1}-1,n_{2,p}; s_{1,q}, m_{1,p})   $

$ + \sum_{i=1}^{p} \sum_{j=1}^{m_i-1} D^{-4}(n_{2,i}, r_{1,q}, r_{q+1}-1, n_{i+1,p}; m_{1,i-1}, j, s_{1,q}, m_i-j, m_{i+1,p})   $

$ +\delta_{r_{q+1},1} E^{-4}(n_{2,p}, r_{1,q}; m_{1,p}, s_{1,q}),  $

\item[ ] $[E^{-4}(n_{2,p};m_{1,p}), G^{-1}(r_{1, q+1};s_{1, q})] =  C^{-4}(1,r_{1,q}, r_{q+1}-1, n_{2,p}; s_{1,q},m_{1,p}) $

$ + \sum_{i=1}^{p}\sum_{j=1}^{m_i-1} E^{-4}(n_{2,i}, r_{1,q}, r_{q+1}-1, n_{i+1,p}; m_{1,i-1},j, s_{1,q},m_i-j, m_{i+1,p}),$

\item   $ [A^{-4}, H^{-1}( ;1)] = 0,$

\item[ ] $[A^{-4}(n), H^{-1}( ;1)] = A^{-4}(n),$

\item[ ] $[B^{-4}(n_{0,p};m_{1,p}), H^{-1}( ;1)] =  - (\sum_{i=1}^{p} m_i-1)B^{-4}(n_{0,p};m_{1,p}),$

\item[ ] $[C^{-4}(n_{0,p};m_{1,p}), H^{-1}( ;1)] =  - (\sum_{i=1}^{p} m_i-1)C^{-4}(n_{0,p};m_{1,p}),$

\item[ ] $[D^{-4}(n_{2,p};m_{1,p}), H^{-1}( ;1)] =  - (\sum_{i=1}^{p} m_i-1)D^{-4}(n_{2,p};m_{1,p}),$

\item[ ] $[E^{-4}(n_{2,p};m_{1,p}), H^{-1}( ;1)] =  - (\sum_{i=1}^{p} m_i-1)E^{-4}(n_{2,p};m_{1,p}),$

\item  the operators $[-, H^{-1}(r_{2, q};s_{1, q})]$  (excluding the case $[-, H^{-1}( ;1)]$) act on  $\HH^{-4}(A)$  in the following way:

\item[ ] $[A^{-4}, H^{-1}(r_{2, q};s_{1, q})] = 0,$

\item[ ] $[A^{-4}(n),  H^{-1}(r_{2, q};s_{1, q})] = 0,$

\item[ ] $[B^{-4}(n_{0,p};m_{1,p}), H^{-1}(r_{2, q};s_{1, q})] $

$= - \sum_{i=1}^{p} \sum_{j=0}^{m_i-1} B^{-4}(n_{0,i}, r_{2,q}, n_{i+1,p}; m_{1,i-1}, s_1+j, s_{2,q-1}, s_q +m_i -j-1, m_{i+1,p}),$

\item[ ] $[C^{-4}(n_{0,p};m_{1,p}), H^{-1}(r_{2, q};s_{1, q})] $

$= - \sum_{i=1}^{p} \sum_{j=0}^{m_i-1} C^{-4}(n_{0,i}, r_{2,q}, n_{i+1,p}; m_{1,i-1}, s_1+j, s_{2,q-1}, s_q +m_i -j-1, m_{i+1,p}),$

\item[ ] $[D^{-4}(n_{2,p};m_{1,p}), H^{-1}(r_{2, q};s_{1, q})] $

$= - \sum_{i=1}^{p} \sum_{j=0}^{m_i-1} D^{-4}(n_{2,i}, r_{2,q}, n_{i+1,p}; m_{1,i-1}, s_1+j, s_{2,q-1}, s_q +m_i -j-1, m_{i+1,p}),$

\item[ ] $[E^{-4}(n_{2,p};m_{1,p}), H^{-1}(r_{2, q};s_{1, q})] $

$= - \sum_{i=1}^{p} \sum_{j=0}^{m_i-1} E^{-4}(n_{2,i}, r_{2,q}, n_{i+1,p}; m_{1,i-1}, s_1+j, s_{2,q-1}, s_q +m_i -j-1, m_{i+1,p}),$

\end{itemize}

}

\section{Technical Details of Proofs} \label{Sect: proof}

In this section, we give an explicit computation of the basis elements for Hochschild (co)homology from Section~\ref{Sect: Hochschild co + homo}, together with an explicit computation of the comparison morphisms, given in Section~\ref{Sect: Comparison morphisms}, between the bar resolution and the Koszul resolution of $A$.

Throughout this appendix, we adopt the following convention for the elements of $A$ according to  Corollary~\ref{cor:k-basis_A}:
\begin{equation}\label{eq: expression of a}
 \begin{array}{rl}
		a  = & a_0 1+ \sum a_1(n_0) y^{n_0} + \sum a_2(n_0) y^{n_0}  x \\
 &  + \sum a_3(n_{0,1}) y^{n_0}  x y^{n_1} +  a_7 x + \sum a_8(n_1) xy^{n_1}   \\
		&  +  \sum a_4(n_{0, p};m_{1, p}) y^{n_0} x y^{n_1} z^{m_1} \cdots z^{m_p} \\
		&  +  \sum a_{5}(n_{0, p};m_{1, p}) y^{n_0} x y^{n_1} z^{m_1} \cdots z^{m_p}x \\
		&  +  \sum a_{6}(n_{0, p+1};m_{1, p}) y^{n_0} x y^{n_1} z^{m_1} \cdots z^{m_p} x y^{n_{p+1}}  \\
		&  +  \sum a_{9}(n_{1, p};m_{1, p}) x y^{n_1} z^{m_1}  \cdots z^{m_p}  \\
		&  +  \sum a_{10}(n_{1, p};m_{1, p}) x y^{n_1} z^{m_1}  \cdots z^{m_p}x  \\
		&  +  \sum a_{11}(n_{1, p+1};m_{1, p}) xy^{n_1} z^{m_1} \cdots z^{m_p} x y^{n_{p+1}}\\
		&  +  \sum a_{12}(n_{1, p};m_{1, p})  y^{n_1} z^{m_1}  \cdots z^{m_p}  \\
		&  +  \sum a_{13}(n_{1, p};m_{1, p})   y^{n_1} z^{m_1} \cdots z^{m_p}x  \\
		&  +  \sum a_{14}(n_{1, p+1};m_{1, p}) y^{n_1} z^{m_1}  \cdots z^{m_p} x y^{n_{p+1}}\\
		& + \sum a_{15}(n_{2, p};m_{1, p}) z^{m_1} \cdots z^{m_p} \\
		& + \sum a_{16}(n_{2, p};m_{1, p}) z^{m_1} \cdots z^{m_p} x\\
		& + \sum a_{17}(n_{2, p+1};m_{1, p}) z^{m_1} \cdots z^{m_p}xy^{n_{p+1}},
	\end{array}
\end{equation}  
	where  each summation ranges over all values of the indices $p\geqslant 1, n_0,n_1,\cdots, n_{p+1}, $ $m_1,\cdots,m_p \geqslant 1$  that explicitly appear in its summand.

To facilitate the detailed calculations, we first state and prove the necessary lemma.

\begin{lem}\label{lem:Koszulness}
    Let $a\in A$, then we have the following implications: 
    \begin{enumerate}[(\roman*)]
        \item\label{Koszulness_a(x+y)} $a(x+y) = 0$ $\implies$ $a = 0$; 
         \item\label{Koszulness_ya} $ya = 0$ $\implies$ $a = 0$;
         \item\label{Koszulness_ax}  $ax = 0$ $\implies$ $a = a'(x+y)$ for some $a'\in A$; 
         \item\label{Koszulness_xa} $xa = 0$ $\implies$ $a = za'$ for some $a'\in A$; 
    \end{enumerate}
\end{lem}
\begin{proof}
    Since $A$ is Koszul, its left Koszul resolution  is
    $K^L_\bullet:=A\otimes V_\bullet$ with the differentials defined as follows. 
$$\begin{array}{l}d^L_{4} :K^L_4\to K^L_3 ,\quad a\otimes ( x+y) xzy \mapsto a( x+y) \otimes xzy, \\
d^L_3:K^L_3\to K^L_2 ,\quad \begin{cases} a\otimes (x+y) xz \mapsto a(x+y) \otimes xz,\\ b\otimes  xzy \mapsto bx\otimes zy ,\end{cases} \\
d^L_2:K^L_2\to K^L_1 ,\quad \begin{cases} a\otimes (x +y)x \mapsto a(x+y) \otimes x,\\ b\otimes xz \mapsto bx\otimes z,\\ c\otimes zy \mapsto cz\otimes y,\end{cases} \\
d^L_{1} :K^L_1\to K^L_0 ,\quad a\otimes p\mapsto ap,\ p=x,y,z.\end{array}$$ 

Suppose $a\in A$ satisfies $a(x+y) = 0.$ It follows that $a\otimes ( x+y) xzy \in \ker d_4^L.$ By the exactness, we have $\ker d_4^L=0,$ which implies that $a=0.$ This gives \ref{Koszulness_a(x+y)}. 

Suppose $b \in A$ satisfies $bx=0$. It follows that $b\otimes xzy\in \ker d_3^L.$ By the exactness, we have $\ker d_4^L=\im d_4^L = \{a(x+y)\otimes xzy\mid a\in A\}.$ 
Therefore $b=a(x+y)$ for some $a\in A.$ This gives \ref{Koszulness_ax}. 

The other two conclusions follow similarly by appealing to the exactness of the right Koszul resolution.
\end{proof} 

Our strategy for computing Hochschild (co)homology proceeds in two main steps.
\begin{itemize}
  \item First, for any element in $\ker d_n$  (resp. $\ker d_n^*$), we simplify its representative by adding an appropriate element from $\im d_{n+1}$ (resp. $\im d_{n-1}^*$ ) without changing its (co)homology class in (co)homology.
  \item Then, the remaining relations imposed by the cycle (resp. cocycle) condition determine an explicit basis for the (co)homology space.
\end{itemize}

\begin{prop}\label{prop:HH_4}
   The Hochschild homology $\HH_{4}(A)$ of $A$ vanishes. 
\end{prop}
\begin{proof}
    Since $d_5=0$, we have $\HH_4(A) = \ker d_4$. For $a\otimes (x+y)xzy\in A\otimes V_4$, we have $d_4(a) = a(x+y)\otimes xzy + ya\otimes (x+y)xz$. Thus $a\in\ker d_4$ is equivalent to the conditions 
    $$        a(x+y) = 0 \quad  \text{and}\quad   ya = 0 .$$
 By Lemma \ref{lem:Koszulness}, either of the two conditions above   implies that $a=0$. Therefore, $\HH_4(A) = 0$.  
\end{proof}

\begin{prop}\label{prop:HH_3}
   The Hochschild homology $\HH_{3}(A)$ of $A$ vanishes. 
\end{prop}
\begin{proof}
Let $\alpha = a\otimes xzy + b\otimes (x+y)xz \in\ker d_3$. From $d_3(\alpha) =0$, we obtain
$$\begin{array}{c}
     ax = 0,   \\ 
     b(x+y)-ya = 0,   \\ 
     zb = 0. 
\end{array}$$
By Lemma~\ref{lem:Koszulness}, $ax=0$ implies $a = a'(x+y)$ for some $a'\in A$. Substituting this expression for $a$ into the second equation gives:
$$ b(x+y)-ya = b(x+y)-ya'(x+y) = (b-ya')(x+y)= 0.$$
By Lemma~\ref{lem:Koszulness} again, we have $b-ya'=0$, i.e., $b=ya'$. It follows that
$$ \ker d_3\subseteq \{a'(x+y)\otimes xzy + ya'\otimes (x+y)xz \mid a'\in A\}.$$
On the other hand, 
$$ \im d_4 = \left\{c( x+y) \otimes xzy+yc\otimes (x+y)xz   \mid c\in A\right\}.$$
Comparing this with the description of $\ker d_3$ obtained earlier, we observe that the two sets are identical. 
Therefore, $\ker d_3 = \operatorname{im} d_4$ and $\HH_3(A) = 0.$
\end{proof}

\begin{prop}\label{prop:HH_2}
   The Hochschild homology $\HH_{2}(A)$ of $A$ vanishes. 
\end{prop}
\begin{proof}
Let $\alpha = c\otimes xz + d\otimes zy + e\otimes (x+y)x \in \ker d_2.$
We first simpify the expression of  $c$ by adding a suitable element of $\im d_3$ to $\alpha$. 
A direct computation shows that  the space $\{ w \in A \mid  u\otimes (x+y)x + v \otimes xz + w\otimes zy \in \im d_3, \text{ for some } u,w\in A   \}$ has the following basis.
\begin{itemize}
  \item $ x,  y^{n_0},  y^{n_0} x, y^{n_0} x y^{n_1},  xy^{n_1}, $
  \item $y^{n_0} x y^{n_1} z^{m_1}   \cdots z^{m_p}, $
  \item $y^{n_0} x y^{n_1} z^{m_1}   \cdots z^{m_p}x, $
  \item $y^{n_0} x y^{n_1} z^{m_1}   \cdots z^{m_p}xy^{n_{p+1}}, $
  \item $ x y^{n_1} z^{m_1}   \cdots z^{m_p}x, $
  \item $ x y^{n_1} z^{m_1}   \cdots z^{m_p}xy^{n_{p+1}}, $
  \item $y^{n_1} z^{m_1}   \cdots z^{m_p}, $
  \item $y^{n_1} z^{m_1}   \cdots z^{m_p}x, $
  \item $y^{n_1} z^{m_1}   \cdots z^{m_p}xy^{n_{p+1}}, $
  \item $ z^{m_1}   \cdots z^{m_p}x, $
  \item $ z^{m_1}   \cdots z^{m_p}xy^{n_{p+1}}, $
\end{itemize}
for $p, n_{0}, \cdots, n_{p+1}, m_{1}, \cdots, m_p \geqslant 1.$
Therefore, we may assume that the expression for $c$ does not contain the above monomials, i.e., we set
 $$ \begin{array}{rl} c = &c_0 \cdot 1 +  \sum c_9(n_{1,p};m_{1,p})xy^{n_1}z^{m_1}\cdots z^{m_p}  + \sum c_{15}(n_{2,p};m_{1,p}) z^{m_1}\cdots z^{m_p},\end{array}$$
where the summations run over all basis elements of $\B$ in Corollary~\ref{cor:k-basis_A}  of the corresponding types.
Now $\alpha\in\ker d_2$ if and only if 
\begin{gather}
     e(x+y) + zc + xe = 0, \label{eq:1st-kerd_2}
     \\ xe + dz = 0, \label{eq:2nd-kerd_2}
     \\ cx + yd  =0. \label{eq:3rd-kerd_2}
\end{gather}
We note that each monomial in $yd$ begins with $y$, while each monomial in $cx$ begins with either $x$ or $z$ under our assumption on $c$. Consequently, equation~\eqref{eq:3rd-kerd_2}  forces $yd=cx=0.$
Since 
$$ \begin{array}{rl}
    cx= & c_0 x + c_9(n_{1,p};m_{1,p})xy^{n_1}z^{m_1}\cdots z^{m_p}x  +c_{15}(n_{2,p};m_{1,p}) z^{m_1}\cdots z^{m_p}x 
     =  0,\\
\end{array}$$
it follows that  $c_0 = c_{9} = c_{15}=0,$ whence $c=0$. 
By Lemma \ref{lem:Koszulness}, $yd =0$ implies $d = 0$. Therefore, we have $c=0$ and $d=0$. 
With this, equations~\eqref{eq:1st-kerd_2} and \eqref{eq:2nd-kerd_2} reduces to $e(x+y) = 0$, which, by Lemma \ref{lem:Koszulness}, forces $e=0$. 
Therefore, we have $\HH_2(A)=0$.  
\end{proof}

The computation of $\HH_1(A)$ proceeds as follows.
\begin{proof}[Proof of Proposition~\ref{prop:HH_1}]
Let $\alpha = a\otimes x + b\otimes y+ c\otimes z \in \ker d_1.$ 
Following an argument similar to that in the previous proof, we first simplify the expressions for  $a,b,$ and $c$.
A direct computation yields the following results: 

 $(i)$  $\{ v \in A \mid  u\otimes x + v \otimes y + w\otimes z \in \im d_2, \text{ for some } u,w\in A   \}$ has the following basis $R_2$:
\begin{itemize}
  \item $y^{n_0}x, y^{n_0}x y^{n_1}, x, xy^{n_1},$
  \item $  y^{n_0} x y^{n_1} z^{m_1} \cdots z^{m_p},$
  \item $  y^{n_0} x y^{n_1} z^{m_1} \cdots z^{m_p}x,$
  \item $  y^{n_0} x y^{n_1} z^{m_1} \cdots z^{m_p}xy^{n_{p+1}},$
  \item $  x y^{n_1} z^{m_1} \cdots z^{m_p},$
  \item $  x y^{n_1} z^{m_1} \cdots z^{m_p}x,$
  \item $  x y^{n_1} z^{m_1} \cdots z^{m_p}xy^{n_{p+1}},$
  \item $  y^{n_1} z^{m_1} \cdots z^{m_p},$
  \item $  z^{m_1} \cdots z^{m_p},$
\end{itemize}
for $p, n_{0}, \cdots, n_{p+1}, m_{1}, \cdots, m_p \geqslant 1.$ 

$(ii)$ $\{ w \in A \mid  u\otimes x  + w\otimes z \in \im d_2,   \text{ for some } u\in A \}$ has the following basis $R_3$:
\begin{itemize}
  \item $y^{n_0}x, y^{n_0}xy^{n_1},x,$
  \item $  y^{n_0} x y^{n_1} z^{m_1} \cdots z^{m_p},$
  \item $  y^{n_0} x y^{n_1} z^{m_1} \cdots z^{m_p}x,$
  \item $  y^{n_0} x y^{n_1} z^{m_1} \cdots z^{m_p}xy^{n_{p+1}},$
  \item $  x y^{n_1} z^{m_1} \cdots z^{m_p}x,$
  \item $  y^{n_1} z^{m_1} \cdots z^{m_p}x,$
  \item $  z^{m_1} \cdots z^{m_p}x,$
\end{itemize}
for $p, n_{0}, \cdots, n_{p+1}, m_{1}, \cdots, m_p \geqslant 1.$ 

$(iii)$ $\{ u \in A \mid  u\otimes x  \in \im d_2   \}$ has the following basis $R_1$:
\begin{itemize}
  \item $ z^{m_1} \cdots z^{m_p}x,$
  \item $ z^{m_1} \cdots z^{m_p}xy^{n_{p+1}},$
\end{itemize}
for $p, n_{2}, \cdots, n_{p+1}, m_{1}, \cdots, m_p \geqslant 1.$

Thus, we may assume that 
$$ \begin{array}{rl}
		a  = & a_0 1+ \sum a_1(n_0) y^{n_0} + \sum a_2(n_0) y^{n_0}  x \\
        &   + \sum a_3(n_{0,1}) y^{n_0}  x y^{n_1} +  a_7 x + \sum a_8(n_1) xy^{n_1}   \\
		&  +  \sum a_4(n_{0, p};m_{1, p}) y^{n_0} x y^{n_1} z^{m_1} \cdots z^{m_p} \\
		&  +  \sum a_{5}(n_{0, p};m_{1, p}) y^{n_0} x y^{n_1} z^{m_1} \cdots z^{m_p}x \\
		&  +  \sum a_{6}(n_{0, p+1};m_{1, p}) y^{n_0} x y^{n_1} z^{m_1} \cdots z^{m_p} x y^{n_{p+1}}  \\
		&  +  \sum a_{9}(n_{1, p};m_{1,  p}) x y^{n_1} z^{m_1}  \cdots z^{m_p}  \\
		&  +  \sum a_{10}(n_{1, p};m_{1, p}) x y^{n_1} z^{m_1}  \cdots z^{m_p}x  \\
		&  +  \sum a_{11}(n_{1, p+1};m_{1, p}) xy^{n_1} z^{m_1} \cdots z^{m_p} x y^{n_{p+1}}\\
		&  +  \sum a_{12}(n_{1, p};m_{1, p})  y^{n_1} z^{m_1}  \cdots z^{m_p}  \\
		&  +  \sum a_{13}(n_{1, p};m_{1, p})   y^{n_1} z^{m_1} \cdots z^{m_p}x  \\
		&  +  \sum a_{14}(n_{1, p+1};m_{1, p}) y^{n_1} z^{m_1}  \cdots z^{m_p} x y^{n_{p+1}}\\
		& + \sum a_{15}(n_{2, p};m_{1, p}) z^{m_1} \cdots z^{m_p},
	\end{array} $$
$$ \begin{array}{rl}
		b  = & b_0 1+ \sum b_1(n_0) y^{n_0}   \\
		&  +  \sum b_{13}(n_{1, p};m_{1, p})   y^{n_1} z^{m_1} \cdots z^{m_p}x  \\
		&  +  \sum b_{14}(n_{1, p+1};m_{1, p}) y^{n_1} z^{m_1}  \cdots z^{m_p} x y^{n_{p+1}}\\
		& + \sum b_{16}(n_{2, p};m_{1, p}) z^{m_1} \cdots z^{m_p} x\\
		& + \sum b_{17}(n_{2, p+1};m_{1, p}) z^{m_1} \cdots z^{m_p}xy^{n_{p+1}},
	\end{array} $$
$$ \begin{array}{rl}
		c  = & c_0 1+ \sum c_1(n_0) y^{n_0}  + \sum c_8(n_1) xy^{n_1}   \\
		&  +  \sum c_{9}(n_{1, p};m_{1, p}) x y^{n_1} z^{m_1}  \cdots z^{m_p}  \\
		&  +  \sum c_{11}(n_{1, p+1};m_{1, p}) xy^{n_1} z^{m_1} \cdots z^{m_p} x y^{n_{p+1}}\\
		&  +  \sum c_{12}(n_{1, p};m_{1, p})  y^{n_1} z^{m_1}  \cdots z^{m_p}  \\
		&  +  \sum c_{14}(n_{1, p+1};m_{1, p}) y^{n_1} z^{m_1}  \cdots z^{m_p} x y^{n_{p+1}}\\
		& + \sum c_{15}(n_{2, p};m_{1, p}) z^{m_1} \cdots z^{m_p} \\
		& + \sum c_{17}(n_{2, p+1};m_{1, p}) z^{m_1} \cdots z^{m_p}xy^{n_{p+1}},
	\end{array} $$

Substituting the above expressions for $a,b,$ and $c$ into 
\begin{equation*}\label{eq: d_1 alpha =0}
d_1(\alpha)=ax-xa+by-yb+cz-zb=0
\end{equation*} 
and comparing coefficients of each basis element  of $A$, we obtain relations among all coefficients of $a,b,$ and $c$ as shown below: 

\begin{enumerate}[(\roman*)]
  \item the coefficient of $1$   is already zero;
  \item the coefficient of $y^{n_0}$   is already   zero; 
  \item from the coefficients of $y^{n_0}x$ and $y^{n_0}xy^{n_1}$, we obtain $a_1=a_3=a_8=0$; 
  \item the coefficient of $x$   is  already zero;
  \item from the coefficient of $xy^{n_1}$,  we obtain  $a_1=0;$
  \item\label{a_4 = a_9 =0} from the coefficient of $y^{n_0}xy^{n_1}z^{m_1}\cdots z^{m_p}$,  we obtain   $a_4=a_9=0;$
  \item from the coefficient of $y^{n_0}xy^{n_1}z^{m_1}\cdots z^{m_p}x$, together with \ref{a_4 = a_9 =0}, we obtain $a_5=a_{10}=0;$
  \item from the coefficient of $y^{n_0}xy^{n_1}z^{m_1}\cdots z^{m_p}xy^{n_{p+1}}$, we obtain $a_6=a_{11}=0;$
  \item from the coefficient of $xy^{n_1}z^{m_1}\cdots z^{m_p}$, we obtain 
  
  $ a_{12}(n_1;1) = c_8(n_1), $
  
  $a_{12}(n_{1,p};m_{1,p-1}, m_{p}+1) = c_9(n_{1,p}; m_{1,p}),  $ 
  
  $a_{12}(n_{1,p+1}; m_{1,p},1) = c_{11}(n_{1,p+1};m_{1,p}),$
  
  \noindent for   $p, n_{1},\cdots,n_{p+1},m_1,\cdots,m_p\geqslant 1;$
  
  \item from the coefficient of $xy^{n_1}z^{m_1}\cdots z^{m_p}x$, together with \ref{a_4 = a_9 =0}, we obtain  $a_{13}=0;$
  \item from the coefficient of $xy^{n_1}z^{m_1}\cdots z^{m_p}xy^{n_{p+1}}$,  we obtain  $a_{14}=0;$
  \item from the coefficient of $y^{n_1}z^{m_1}\cdots z^{m_p}$,  we obtain  $c_1=c_{12}= c_{14} = 0;$
  \item from the coefficient of $y^{n_1}z^{m_1}\cdots z^{m_p}x$,  we obtain  
  
  $a_{12}(1,n_{2,p}; m_{1,p}) = b_{16}(n_{2,p};m_{1,p}),$
  
  $a_{12}(n_1+1,n_{2,p};m_{1,p}) = b_{13}(n_{1,p}; m_{1,p}),$
  
  \noindent for   $p, n_{1},\cdots,n_{p},m_1,\cdots,m_p\geqslant 1;$
  
  \item from the coefficient of $y^{n_1}z^{m_1}\cdots z^{m_p}xy^{n_{p+1}}$,  we obtain  
  
  $ b_{13}(1,n_{2,p}; m_{1,p}) = b_{17}(n_{2,p},1; m_{1,p}),  $ 
  
  $ b_{13}(n_1+1, n_{2,p}; m_{1,p}) = b_{14}(n_{1,p},1; m_{1,p}),$
  
  $b_{14}(1,n_{2,p+1};m_{1,p}) = b_{17}(n_{2,p},n_{p+1}+1; m_{1,p}),$
  
  $b_{14}(n_1+1, n_{2,p+1}; m_{1,p}) = b_{14}(n_{1,p},n_{p+1}+1;m_{1,p}), $
  
  \noindent for   $p, n_{1},\cdots,n_{p+1},m_1,\cdots,m_p\geqslant 1;$
  
  \item from the coefficient of $z^{m_1}\cdots z^{m_p}$,  we obtain  
  
  $ c_{17}(n_{2,p+1}; 1,m_{2,p}) = c_9(n_{2,p+1}; m_{2,p},1),  $
  
  $ c_{17}(n_{2,p+1}; m_1+1, m_{2,p}) = c_{15}(n_{2,p+1}; m_{1,p},1),  $
  
  $ c_{15}(n_{2,p+1}; 1, m_{2,p+1}) = c_9(n_{2,p+1}; m_{2,p},m_{p+1}+1),$
  
  $ c_{15}(n_{2,p+1};m_1+1,m_{2,p+1}) = c_{15}(n_{2,p+1};m_{1,p},m_{p+1}+1), $
  
  \noindent for   $p, n_{1},\cdots,n_{p+1},m_1,\cdots,m_{p+1}\geqslant 1;$
  
  \item from the coefficient of $z^{m_1}\cdots z^{m_p}x$,  we obtain   $a_{15}=0;$
  
  \item from the coefficient of $z^{m_1}\cdots z^{m_p}xy^{n_{p+1}}$,  we obtain    
  
  $ b_{16}(; 1) = c_8(1), $
  
  $ b_{17}(n_2;1) = c_8(n_2+1),$
  
  $ b_{16}(n_{2,p+1}; 1,m_{2,p+1}) = c_{11}(n_{2,p+1},1;m_{1,p+1}),$ 
  
  $ b_{17}(n_{2,p+2}; 1,m_{2,p+1}) = c_{11}(n_{2,p+1}, n_{p+2}+1; m_{2,p+1}),$
  
  $ b_{16}(n_{2,p}; m_1+1,m_{2,p}) = c_{17}(n_{2,p},1; m_{1,p}), $
  
  $ b_{17}(n_{2,p+1}; m_1+1,m_{2,p}) = c_{17}(n_{2,p},n_{p+1}+1;m_{1,p}),$
  
  \noindent for   $p, n_{1},\cdots,n_{p+2},m_1,\cdots,m_{p+1}\geqslant 1.$
\end{enumerate} 
Therefore, $a,b,$ and $c$ have the following expressions: 
$$ \begin{array}{rl}
		a  = & a_0 1 + \sum a_2(n_0) y^{n_0}  x +  a_7 x   \\
		&  +  \sum a_{12}(n_{1, p};m_{1, p})  y^{n_1} z^{m_1}  \cdots z^{m_p},\\
		b  = & b_0 1+ \sum b_1(n_0) y^{n_0}   \\
		&  +  \sum b_{13}(n_{1, p};m_{1, p})   y^{n_1} z^{m_1} \cdots z^{m_p}x  \\
		&  +  \sum b_{14}(n_{1, p+1};m_{1, p}) y^{n_1} z^{m_1}  \cdots z^{m_p} x y^{n_{p+1}}\\
		& + \sum b_{16}(n_{2, p};m_{1, p}) z^{m_1} \cdots z^{m_p} x\\
		& + \sum b_{17}(n_{2, p+1};m_{1, p}) z^{m_1} \cdots z^{m_p}xy^{n_{p+1}},\\
		c  = & c_0 1  + \sum c_8(n_1) xy^{n_1}   \\
		&  +  \sum c_{9}(n_{1, p};m_{1, p}) x y^{n_1} z^{m_1}  \cdots z^{m_p}  \\
		&  +  \sum c_{11}(n_{1, p+1};m_{1, p}) xy^{n_1} z^{m_1} \cdots z^{m_p} x y^{n_{p+1}}\\
		& + \sum c_{15}(n_{2, p};m_{1, p}) z^{m_1} \cdots z^{m_p} \\
		& + \sum c_{17}(n_{2, p+1};m_{1, p}) z^{m_1} \cdots z^{m_p}xy^{n_{p+1}},
	\end{array} $$
where the coefficients $a_0,a_2,a_7,b_0,b_1,c_0,c_{15}(;m_1)$ are free, and among the remaining terms, the coefficients within each of the two classes listed below are respectively equal.
\begin{itemize}
  \item  the term $y^{n_1}z$ in $a$, 
  
  \noindent  the terms $y^{i}zxy^{n_1-i-1}$ in $b$, for $0\leqslant i\leqslant n_1-1,$
  
  \noindent   the term  $xy^{n_1}$ in $c$;
  \item  the terms $y^{n_i}z^{m_i}\cdots z^{m_p} xy^{n_1}z^{m_1}\cdots z^{m_{i-1}}$ in $a$, for $ 1\leqslant i\leqslant p,$ 
  
  \noindent  the terms $y^{j}z^{m_i}\cdots z^{m_p} xy^{n_1}z^{m_1}\cdots z^{m_{i-1}}xy^{n_i-j-1}$ in $b$,  for $ 1\leqslant i\leqslant p, 1\leqslant j\leqslant n_i-1,$
  
  \noindent  the terms $z^{j}xy^{n_{i+1}} z^{m_{i+1}}\cdots z^{m_p} xy^{n_1}z^{m_1}\cdots z^{m_{i-1}}xy^{n_i} z^{m_i-j-1}$ in $c$ , for $ 1\leqslant i\leqslant p, 1\leqslant j\leqslant m_i-1$.

\end{itemize}
Consequently, $\alpha$ must be a linear combination of the following mutually independent elements:
$$1\ot x, \  y^{n-1}x \ot x, \ y^{n-1} \ot y,\  z^{n-1} \ot z,\  \sum_{uvw= xy^{n_1}z^{m_1}\cdots z^{m_p}} wu \otimes v,$$ 
for $n,p, n_1,\cdots, n_p,m_1,\cdots, m_p\geqslant 1.$  
Since we assume that $a,b,$ and $c$ in $\alpha$ contain no terms corresponding  respectively  to the monomials in $R_1,R_2,$ and $R_3$,  the above elements remain linearly independent modulo $ \im d_2.$
Therefore, the set of their homology classes gives a basis for  $\HH_{1}(A)$, precisely as listed.
\end{proof}

The computation of $\HH_0(A)$ proceeds as follows.
\begin{proof}[Proof of Proposition~\ref{prop:HH_0}]
It is easy to see that the elements listed in the proposition are linearly independent in $\HH_0(A)$; 
therefore, it suffices to verify that every element of $A$ can be written as a linear combination of these elements. 
According to Remark~\ref{rmk:HH_0}, monomials in $A$ satisfy that moving a letter from one side to the other does not change their class in $\HH_0(A)$.
Using this fact, we give the images of the basis elements $\B$ of $A$ in $\HH_0(A)$ as follows:
   \begin{itemize}
     \item  $1=\zeta_0,$
     \item $y^{n_0} = \beta_0(n_0),$
     \item  $y^{n_0}x = (-1)^{n_0} x^{n_0+1} =  (-1)^{n_0} \alpha_0(n_0+1),$
     \item $ y^{n_0}xy^{n_1} = y^{n_0+n_1}x =(-1)^{n_0+n_1} x^{n_0+n_1+1} =  (-1)^{n_0+n_1} \alpha_0(n_0+n_1+1),$
     \item $x=\alpha_0(1),$
     \item $xy^{n_1} = y^{n_1}x = (-1)^{n_1}x^{n_1+1} = (-1)^{n_1} \alpha_0(n_1+1),$
     \item $y^{n_0} x y^{n_1} z^{m_1}   \cdots z^{m_p} = y^{n_0-1} x y^{n_1} z^{m_1}   \cdots z^{m_p}y = 0,$
     \item $y^{n_0} x y^{n_1} z^{m_1}   \cdots z^{m_p}x = y^{n_1} z^{m_1}   \cdots z^{m_p}xy^{n_0} x   = 0,   $
     \item $ y^{n_0} x y^{n_1} z^{m_1}   \cdots z^{m_p}x y^{n_{p+1}} =y^{n_1} z^{m_1}   \cdots z^{m_p}x y^{n_{p+1}+n_0} x =0, $
     \item $ x y^{n_1} z^{m_1}   \cdots z^{m_p} = \overline \epsilon_0(n_{1,p}; m_{1,p}), $
     \item  $x y^{n_1} z^{m_1}   \cdots z^{m_p}x =  y^{n_1} z^{m_1}   \cdots z^{m_p}xx = 0, $
     \item $x y^{n_1} z^{m_1}   \cdots z^{m_p}xy^{n_{p+1}} =  y^{n_1} z^{m_1}   \cdots z^{m_p}xy^{n_{p+1}}x = 0,  $
     \item $ y^{n_1} z^{m_1}   \cdots z^{m_p} = y^{n_1-1}z^{m_1}   \cdots z^{m_p}y = 0, $ 
     \item $y^{n_1} z^{m_1}   \cdots z^{m_p} x = xy^{n_1} z^{m_1}   \cdots z^{m_p} = \overline \epsilon_0(n_{1,p}; m_{1,p}),  $
     \item $y^{n_1} z^{m_1}   \cdots z^{m_p} x y^{n_{p+1}} = x y^{n_+n_{p+1}} z^{m_1}   \cdots z^{m_p}  = \overline \epsilon_0(n_{1}+n_{p+1},n_{2,p}; m_{1,p}), $
     \item $  z^{m_1}   \cdots z^{m_p} = xy^{n_2}z^{m_2}\cdots z^{m_{p-1}} xy^{n_p}z^{m_1+m_p} = \overline \epsilon_0(n_{2,p}; m_{2,p-1}, m_1+m_p),  $
     \item $z^{m_1}   \cdots z^{m_p}x = xz^{m_1}   \cdots z^{m_p} =0, $
     \item $z^{m_1}   \cdots z^{m_p}x y^{n_{p+1}} = x y^{n_{p+1}} z^{m_1}   \cdots z^{m_p} = \overline \epsilon_0(n_{p+1},n_{2,p}; m_{1,p}),$
   \end{itemize}
   for   positive integers  $p,n_0,\cdots, n_{p+1},$ and $ m_1, \cdots, m_p.$ 
    Among the relations above, those equal to zero follow from $xz = zxy^n x = zy = 0$ in $A$, for $n\geqslant 0.$
\end{proof}

\bigskip

The computation of $\HH^0(A)$ proceeds as follows.
\begin{proof}[Proof of Proposition~\ref{prop: HH^0}]
 By definition, we have   $\HH^0(A) =\ker d_1^*$.  
 Now take an arbitrary element $a\in \ker(d_1^{*}).$ 
The condition  $d_1^{*}(a)   = 0 $ implies that $a$ commutes with $x,y$ and $z$:
\begin{gather*}
  xa-ax  = 0,  \label{eq: 1st d_1*}   \\ 
  ya-ay  = 0, \label{eq: 2nd d_1*} \\ 
  za-az  = 0.   \label{eq: 3rd d_1*} 
   \end{gather*} 
Expanding $a$ in the basis $\B$ and equating coefficients in the preceding equation, we obtain 
 $a= a_0 \cdot  1.$
Hence $\HH^{0}(A) = \ker(d_1^{*}) = \{ a_0 \mid a_0 \in \bfk \}\cong \bfk.$
\end{proof}

The computation of $\HH^{-1}(A)$ proceeds as follows.
\begin{proof}[Proof of Proposition~\ref{prop: HH^1}]
Let $f\in \ker d_{2}^{*},$ we write $f = (x \mapsto a, y \mapsto b, z \mapsto c).$
Analogous to the homological case, we modify $f$ by a suitable element of $\im d_1^*$  to simplify  the element  $a,b,$ and $c.$
A direct computation shows that the vector space $\{\phi(x) \mid \phi \in \im d_1^* \}$  has a basis $R_1$  consisting of: 
  \begin{itemize}
    \item $xy^{n_0} - y^{n_0}x, $
\item $y^{n_0}x y^{n_1} - y^{n_0+n_1}x,$
\item $y^{n_0}xy^{n_1}z^{m_1}\cdots z^{m_p},$
\item $y^{n_0}xy^{n_1}z^{m_1}\cdots z^{m_p}x,$
\item $y^{n_0}xy^{n_1}z^{m_1}\cdots z^{m_p}xy^{n_{p+1}},$
\item $xy^{n_1}z^{m_1}\cdots z^{m_p} - y^{n_1}z^{m_1}\cdots z^{m_p}x,$
\item $xy^{n_1}z^{m_1}\cdots z^{m_p}x, $
\item $xy^{n_1}z^{m_1}\cdots z^{m_p}xy^{n_{p+1}}, $
\item $z^{m_1} \cdots z^{m_p}x, $
  \end{itemize}
for $p, n_{0}, \cdots, n_{p+1}, m_{1}, \cdots, m_p \geqslant 1.$

 The vector space $\{\phi(y) \mid \phi \in \im d_1^*, \phi(x)=0 \}$ has a basis $R_2$  consisting of: 
  \begin{itemize}
  \item $y^{n_1}x - y^{n_1-1}xy,$
  \item $z^{m_1} \cdots  z^{m_p}xy  -  yz^{m_1} \cdots z^{m_p}x,$
  \item $ z^{m_1} \cdots   z^{m_p}xy^{n_{p+1}+1}- yz^{m_1} \cdots  z^{m_p}xy^{n_{p+1}},$
\end{itemize}
 for $p, n_{0}, \cdots, n_{p+1}, m_{1}, \cdots, m_p \geqslant 1.$
 
 The vector space $\{\phi(z) \mid \phi \in \im d_1^*, \phi(x)=\phi(y)=0 \} $ is trivial.

Thus, we may assume that 
$$ \begin{array}{rl}
		a  = & a_0 1+ \sum a_1(n_0) y^{n_0} + \sum a_2(n_0) y^{n_0}  x    +  a_7 x   \\
		&  +  \sum a_{12}(n_{1, p};m_{1, p})  y^{n_1} z^{m_1}  \cdots z^{m_p}  \\
		&  +  \sum a_{13}(n_{1, p};m_{1, p})   y^{n_1} z^{m_1} \cdots z^{m_p}x  \\
		&  +  \sum a_{14}(n_{1, p+1};m_{1, p}) y^{n_1} z^{m_1}  \cdots z^{m_p} x y^{n_{p+1}}\\
		& + \sum a_{15}(n_{2, p};m_{1, p}) z^{m_1} \cdots z^{m_p} \\
		& + \sum a_{17}(n_{2, p+1};m_{1, p}) z^{m_1} \cdots z^{m_p}xy^{n_{p+1}},
	\end{array} $$
$$ \begin{array}{rl}
		b  = & b_0 1+ \sum b_1(n_0) y^{n_0}  \\
        &  + \sum b_3(n_{0,1}) y^{n_0}  x y^{n_1} +  b_7 x + \sum b_8(n_1) xy^{n_1}   \\   
		&  +  \sum b_4(n_{0, p};m_{1, p}) y^{n_0} x y^{n_1} z^{m_1} \cdots z^{m_p} \\
		&  +  \sum b_{5}(n_{0, p};m_{1, p}) y^{n_0} x y^{n_1} z^{m_1} \cdots z^{m_p}x \\
		&  +  \sum b_{6}(n_{0, p+1};m_{1, p}) y^{n_0} x y^{n_1} z^{m_1} \cdots z^{m_p} x y^{n_{p+1}}  \\
		&  +  \sum b_{9}(n_{1, p};m_{1, p}) x y^{n_1} z^{m_1}  \cdots z^{m_p}  \\
		&  +  \sum b_{10}(n_{1, p};m_{1, p}) x y^{n_1} z^{m_1}  \cdots z^{m_p}x  \\ 
		&  +  \sum b_{11}(n_{1, p+1};m_{1, p}) xy^{n_1} z^{m_1} \cdots z^{m_p} x y^{n_{p+1}}\\
		&  +  \sum b_{12}(n_{1, p};m_{1, p})  y^{n_1} z^{m_1}  \cdots z^{m_p}  \\
		&  +  \sum b_{13}(n_{1, p};m_{1, p})   y^{n_1} z^{m_1} \cdots z^{m_p}x  \\
		&  +  \sum b_{14}(n_{1, p+1};m_{1, p}) y^{n_1} z^{m_1}  \cdots z^{m_p} x y^{n_{p+1}}\\
		& + \sum b_{15}(n_{2, p};m_{1, p}) z^{m_1} \cdots z^{m_p} \\
		& + \sum b_{16}(n_{2, p};m_{1, p}) z^{m_1} \cdots z^{m_p} x,
	\end{array} $$
and $c$ adopts the expression defined in~\eqref{eq: expression of a}.

Since $f \in  \ker d_{2}^{*} $ is a  $1$-cocycle, we have the following relations.
\begin{gather}
  (x+y)a + ax +bx  = 0,  \label{eq: 1st d_2*}   \\ 
   xc + az  = 0, \label{eq: 2nd d_2*} \\ 
   zb + cy = 0.   \label{eq: 3rd d_2*} 
   \end{gather}
Substituting the expressions for $a,b,$ and $c$ into the three equations above, we obtain the following relations among their coefficients: 

Equation \eqref{eq: 1st d_2*} implies that 
\begin{enumerate}[(\roman*)]
  \item the coefficients  $a_0,a_1, a_{12},a_{14},a_{15},a_{17}  $ in $a$ are zero,
  \item the coefficients $b_0,b_4, b_{12}(1, n_{2,p};m_{1,p}) ,b_{15} $ in $b$ are zero,
  \item  $b_1(1) - b_7 - a_7 =0,$
  \item $a_2(n_0) - b_1(n_0+1) + \sum_{i+j=n_0}b_3(i,j) + b_8(n_0) = 0,$
  \item $a_{13}(n_{1,p}; m_{1,p}) = b_9(n_{1,p}; m_{1,p}) = b_{12}(n_1+1,n_{2,p}; m_{1,p}),$
\end{enumerate}
 for $p, n_{0}, \cdots, n_{p}, m_{1}, \cdots, m_p \geqslant 1.$

Using the results just obtained, equation \eqref{eq: 2nd d_2*} yields   $xc=0.$
Therefore, by Proposition~\ref{lem:Koszulness}, we have $c=zc'$ for some $c'\in A$. Consequently, $c$ admits the following expression: 
$$ \begin{array}{rl}
		c  = &   \sum b_{15}(n_{2, p};m_{1, p}) z^{m_1} \cdots z^{m_p} \\
		     & + \sum b_{16}(n_{2, p};m_{1, p}) z^{m_1} \cdots z^{m_p} x\\
             & + \sum b_{17}(n_{2, p+1};m_{1, p}) z^{m_1} \cdots z^{m_p} xy^{n_{p+1}},
	\end{array} $$

From the results obtained above, the third equation \eqref{eq: 3rd d_2*} gives 
\begin{enumerate}[(\roman*)]
\item the coefficients  $a_{13}$ in $a$ are zero,
\item  the coefficients $b_7,b_9, b_{10},b_{12}, b_{16}$ in $b$ are zero,
\item   the coefficients $c_{16}(n_{2,p};m_1+1,m_{2,p}), c_{17}(n_{2,p+1};m_1+1,m_{2,p})$  in $c$ are zero,
\item  $b_8(1) + c_{16}(;1) =0,$
\item $b_8(n_1+1) + c_{17}(n_1;1) = 0,$
\item $b_{11}(n_{1,p},1; m_{1,p}) + c_{16}(n_{1,p}; 1,m_{1,p}) = 0,$
\item $b_{11}(n_{1,p}, n_{p+1}+1; m_{1,p}) + c_{17}(n_{1,p+1}; 1,m_{1,p}) = 0, $
\end{enumerate}
 for $p, n_{0}, \cdots, n_{p+1}, m_{1}, \cdots, m_p \geqslant 1.$

In conclusion, we obtain the following expressions for $a,b,$ and $c$:
$$ \begin{array}{rl}
		a  = &  \sum a_2(n_0) y^{n_0}  x    +  a_7 x,\\
		b  = &   \sum b_1(n_0) y^{n_0}  + \sum b_3(n_{0,1}) y^{n_0}  x y^{n_1}   + \sum b_8(n_1) xy^{n_1}   \\   
		&  +  \sum b_{5}(n_{0, p};m_{1, p}) y^{n_0} x y^{n_1} z^{m_1} \cdots z^{m_p}x \\
		&  +  \sum b_{6}(n_{0, p+1};m_{1, p}) y^{n_0} x y^{n_1} z^{m_1} \cdots z^{m_p} x y^{n_{p+1}}  \\
		&  +  \sum b_{11}(n_{1, p+1};m_{1, p}) xy^{n_1} z^{m_1} \cdots z^{m_p} x y^{n_{p+1}}\\
		&  +  \sum b_{13}(n_{1, p};m_{1, p})   y^{n_1} z^{m_1} \cdots z^{m_p}x  \\
		&  +  \sum b_{14}(n_{1, p+1};m_{1, p}) y^{n_1} z^{m_1}  \cdots z^{m_p} x y^{n_{p+1}}, \\
		c  = &   \sum b_{15}(n_{2, p};m_{1, p}) z^{m_1} \cdots z^{m_p} \\
		     & + \sum b_{16}(n_{1, p};1,m_{1, p}) zxy^{n_1}z^{m_1} \cdots z^{m_p} x\\
             & + \sum b_{17}(n_{1, p+1};1,m_{1, p}) zxy^{n_1}z^{m_1} \cdots z^{m_p} xy^{n_{p+1}},
	\end{array} $$
where the coefficients $b_5,b_{6}, b_{13}, b_{14}, c_{15}$ are free, and the remaining coefficients satisfy the relations
\begin{enumerate}[(\roman*)]
  \item  $b_{1}(1) = a_{7},$
  \item $a_2(n_0) - b_1(n_0+1) + \sum_{i+j=n_0}b_3(i,j) + b_8(n_0) = 0,$
  \item $b_8(1) + c_{16}(;1) = 0,$
  \item $b_8(n_1+1) + c_{17}(n_1;1) = 0,$
  \item $b_{11}(n_{1,p},1; m_{1,p}) + c_{16}(n_{1,p};1,m_{1,p}) = 0,$
  \item $b_{11}(n_{1,p},n_{p+1}+1; m_{1,p}) + c_{17}(n_{1,p+1};1,m_{1,p})=0, $
\end{enumerate}
 for $p, n_{0}, \cdots, n_{p+1}, m_{1}, \cdots, m_p \geqslant 1.$
Therefore, $f$ can be expressed as a linear combination of the  maps listed in the proposition.
These maps remain linearly independent modulo $\im d_1^*,$  since we assume that   $a= f(x)$ and $b= f(y)$   contain no terms corresponding    to the monomials in $R_1$ and $R_2.$
\end{proof}

The computation of $\HH^{-2}(A)$ proceeds as follows.
\begin{proof}[Proof of Proposition~\ref{prop: HH^2}]
Let $f\in \ker d_{2}^{*},$ we write $f = ((x+y)x \mapsto a, xz \mapsto b, zy \mapsto c).$
As in the computation of  $\HH^{-1}(A)$, we begin by simplifying $a,b,$ and $c$ by adding to $f$ an appropriate element of  $\im d_2^{*}.$
Firstly, a direct computation shows that the vector space   $\{\phi(xz) \mid \phi \in \im d_2^*\}$   has a basis $R_2$ consisting of:  
\begin{itemize}
  \item $x, xy^{n_0}, y^{n_0}x, y^{n_0}xy^{n_1},$
  \item $y^{n_0}xy^{n_1}z^{m_1}\cdots z^{m_p},$
  \item $y^{n_0}xy^{n_1}z^{m_1}\cdots z^{m_p}x,$
  \item $y^{n_0}xy^{n_1}z^{m_1}\cdots z^{m_p}xy^{n_{p+1}},$
  \item $xy^{n_1}z^{m_1}\cdots z^{m_p},$
  \item $xy^{n_1}z^{m_1}\cdots z^{m_p}x,$
  \item $xy^{n_1}z^{m_1}\cdots z^{m_p}xy^{n_{p+1}},$
  \item $y^{n_1}z^{m_1}\cdots z^{m_p},$
  \item $z^{m_1} \cdots z^{m_p},$
\end{itemize}
 for $p, n_{0}, \cdots, n_{p+1}, m_{1}, \cdots, m_p \geqslant 1.$
 
  The  vector space  $\{\phi(zy)\mid \phi \in \im d_2^*,  \phi(xz)=0 \}$  has a basis $R_3$ consisting of:  
  \begin{itemize}
  \item  $z^{m_1} \cdots z^{m_p}, $
  \item  $z^{m_1} \cdots z^{m_p}x, $
  \item  $z^{m_1} \cdots z^{m_p}xy^{n_{p+1}}; $
\end{itemize}
 for $p, n_{2}, \cdots, n_{p+1}, m_{1}, \cdots, m_p \geqslant 1.$

 The vector space $\{\phi((x+y)x) \mid \phi \in \im d_2^*,  \phi(xz) = \phi(zy)=0 \}$ has a basis $R_1$ consisting of:  
\begin{itemize}
  \item $ y^{n_0}x, $
  \item $y^{n_0}xy^{n_1}z^{m_1}\cdots z^{m_p}x, $
  \item $xy^{n_1}z^{m_1}\cdots z^{m_p}x, $
  \item $y^{n_1}z^{m_1}\cdots z^{m_p}x,$
\end{itemize}
for $p, n_{0}, \cdots, n_{p+1}, m_{1}, \cdots, m_p \geqslant 1.$
 
 Thus, we may assume that  $a,b,c$  are linear combinations of basis elements of $A$ given in Corollary~\ref{cor:k-basis_A}, excluding those belonging to $R_1, R_2, R_3,$ respectively. 
 
Since $f \in  \ker d_{3}^{*} $ is a  $2$-cocycle, we have the following relations.
$$(x+y)b - az =0 \quad \text{and} \quad xc-by =0.  $$
Substituting the expressions for $a,b,$ and $c$ into the equations above and comparing coefficients of corresponding terms, we deduce that 
$f$ is a linear combination of   elements of types $A^{-2}$ and $B^{-2}$ listed in the proposition.
\end{proof}

\begin{prop}\label{prop: HH^3}
     The Hochschild cohomology $\HH^{-3}(A)$ of $A$ vanishes. 
\end{prop}
\begin{proof}
Let $f\in \ker d_{3}^{*},$ we write $f = ((x+y)xz \mapsto a, xzy \mapsto b).$
A direct computation shows that  the vector spaces $\{\phi(xzy) \mid \phi \in \im d_3^*\},$ has the basis $R_2$ consisting of:
\begin{itemize}
  \item  $ y^{n_0}, y^{n_0}x, y^{n_0}x y^{n_1}, x, x y^{n_1}, $
  \item  $y^{n_0}xy^{n_1}z^{m_1}\cdots z^{m_p}, $
  \item  $y^{n_0}xy^{n_1}z^{m_1}\cdots z^{m_p}x, $
  \item  $y^{n_0}xy^{n_1}z^{m_1}\cdots z^{m_p}xy^{n_{p+1}}, $
  \item  $xy^{n_1}z^{m_1}\cdots z^{m_p}, $
  \item  $xy^{n_1}z^{m_1}\cdots z^{m_p}x, $
  \item  $xy^{n_1}z^{m_1}\cdots z^{m_p}xy^{n_{p+1}}, $
  \item  $y^{n_1}z^{m_1}\cdots z^{m_p}xy^{n_{p+1}}, $
  \item  $z^{m_1}\cdots z^{m_p}xy^{n_{p+1}}, $
\end{itemize}
for $p, n_{0}, \cdots, n_{p+1}, m_{1}, \cdots, m_p \geqslant 1.$ 

The vector spaces $\{\phi((x+y)xz)  \mid \phi \in \in \im d_3^*, \phi(xzy)=0 \},$has   basis $ R_1$ consisting of:  
\begin{itemize}
  \item  $ y^{n_0},y^{n_0}x y^{n_1},x y^{n_1}, $ 
  \item  $y^{n_0}xy^{n_1}z^{m_1}\cdots z^{m_p}xy^{n_{p+1}},$ 
  \item  $xy^{n_1}z^{m_1}\cdots z^{m_p}xy^{n_{p+1}},$ 
  \item  $y^{n_1}z^{m_1}\cdots z^{m_p}xy^{n_{p+1}},$ 
  \item  $z^{m_1}\cdots z^{m_p}xy^{n_{p+1}}, $ 
\end{itemize}
for $p, n_{0}, \cdots, n_{p+1}, m_{1}, \cdots, m_p \geqslant 1.$ 

Thus, we may assume that  $a,b$  are linear combinations of basis elements of $A$ given in Corollary~\ref{cor:k-basis_A}, excluding those belonging to $R_1, R_2,$ respectively. 

Since $f \in  \ker d_{4}^{*} $ is a  $3$-cocycle, we have the following relations.
$$(x+y)b + ay =0.  $$
Substituting the expressions for $a,b$ into the equations above and comparing coefficients of corresponding terms, we deduce that $f=0.$
\end{proof}

The computation of $\HH^{-4}(A)$ proceeds as follows.
\begin{proof}[Proof of Proposition~\ref{prop: HH^4}] 
Since $d_5^{*} = 0,$ the computation of $\HH^{-4}(A)$ reduces to computing a basis for $\im d_4^{*}$.
A direct computation shows that a basis for the image of   $\im d_4^{*}$ on $(x+y)xzy$ is given by the following elements: 
\begin{itemize}
  \item  $ y^{n_0},   y^{n_0}x y^{n_1}, x, x y^{n_1}, $
  \item  $y^{n_0}xy^{n_1}z^{m_1}\cdots z^{m_p}xy^{n_{p+1}},$
  \item  $xy^{n_1}z^{m_1}\cdots z^{m_p}xy^{n_{p+1}},$
  \item  $y^{n_1}z^{m_1}\cdots z^{m_p}xy^{n_{p+1}},$
  \item  $z^{m_1}\cdots z^{m_p}xy^{n_{p+1}}, $
  \item  $yz^{m_1}\cdots z^{m_p}, $
  \item  $yz^{m_1}\cdots z^{m_p}x, $
  \item  $y^{n_1+1}z^{m_1}\cdots z^{m_p}+x y^{n_1}z^{m_1}\cdots z^{m_p}, $
  \item  $y^{n_1+1}z^{m_1}\cdots z^{m_p}x  +x y^{n_1}z^{m_1}\cdots z^{m_p}x, $
\end{itemize}
for $p, n_0,\cdots,n_{p+1},m_1,\cdots,m_p\geqslant 1.$
Then after quotienting by  $\im d_4^{*}$,  $f\in \ker d_{5}^{*} = \Hom(V_4,A)$ must be a linear combination of the images of $(x+y)xzy$  under the maps listed in the proposition.
\end{proof}

\bigskip

Finally,  we   compute the comparison morphisms between the two-sided Koszul resolution   and the two-sided bar resolution via the Morse matching constructed in Section~\ref{Sect: Comparison morphisms}.
In the calculations that follow, we will use the weighted quiver $\overline{Q_{B}}$ from \cite{CLZ24},
which is derived from $Q_B$ by splitting any arrow weighted by a polynomial into multiple parallel arrows, each carrying a single monomial weight.
For instance, the arrow $  (x,x) \xrightarrow{x \ot 1  + 1\ot x}  x $ in  $Q_B$ is replaced in $\overline{Q_{B}}$ with two paralled arrows
$$ d_2^{0}:  (x,x) \xrightarrow{x\ot 1} x  \quad \text{and} \quad  d_2^2: (x,x) \xrightarrow{1\ot x} x.  $$
Furthermore, the dashed arrows in $\overline{Q_B}^{\widetilde{\M}}$ induced by the three types of arrows of $\M'$ are denoted by $d_{n+2}^{-2}$, $d_{n+3}^{-2}$, and $d_{n+4}^{-2}$, respectively.

	\begin{proof}[Proof of Theorem~\ref{thm:Comparision from K to B}]
		We need only compute the even-length zigzag paths in $\overline{Q_B}^{\widetilde{\M}}$ starting from the critical vertices of $\widetilde{\M}$. 
A direct calculation shows that all even-length zigzag paths starting from the critical vertices of $\mathcal{M}'$ are trivial, except for the following cases.
		$$ \begin{tikzcd}[row sep=1.5em,column sep =3em]
			(x,x)  \ar[rrd,  "d_{2}^{1}{,} 1\otimes  1" near end] & & \\
			&	&yx  \ar[lld, dashrightarrow,   " d_{2}^{-1}{,} 1\otimes  1"' near end]  \\
		(y,x)& &
		\end{tikzcd},    $$
		$$ \begin{tikzcd}[row sep=1.5em,column sep =3em]
			(x,x,z)  \ar[rrd,  "d_{3}^{1}{,} 1\otimes  1" near end] & & \\
			&	&(yx,z)  \ar[lld, dashrightarrow,   " d_{3}^{-1}{,} 1\otimes  1"' near end]  \\
		(y,x,z)& &
		\end{tikzcd},    $$
		$$ \begin{tikzcd}[row sep=1.5em,column sep =3em]
			(x,x,z,y)  \ar[rrd,  "d_{4}^{1}{,} 1\otimes  1" near end] & & \\
			&	&(yx,z,y)  \ar[lld, dashrightarrow,   " d_{4}^{-1}{,} 1\otimes  1"' near end]  \\
		(y,x,z,y)& &
		\end{tikzcd}.   $$
	\end{proof}

\begin{proof}[Proof of Theorem~\ref{thm:Comparision from B to K}]
  It suffices to consider even-length zigzag paths in $\overline{Q_B}^{\widetilde{\M}}$ that  end at the critical vertices of $\widetilde{\M}$.
For this, we turn to $\overline{Q_{B}}^{\M}$, where we observe  that for any thick arrow $(w_1,\cdots,w_n) \to (v_1,\cdots,v_{n-1})$, one of the following holds:
\begin{itemize} 
	\item[(i)] 	$v_1 \cdots v_{n-1}$ is a (not necessarily proper) subword of $w_1 \cdots w_n$; 
	\item[(ii)]  $w_1\cdots w_n $ \textbf{reduces to} an expression involving $v_1 \cdots v_{n-1}$.  That is,
	we have $w_1\cdots w_n = v_1 \cdots v_{n-1} + \sum_i c_i m_i$ in $A$, with both $v_1 \cdots v_{n-1}$ and each $m_i$ being monomials that are smaller than $w_1\cdots w_n$ under the given monomial order.
	\end{itemize}
Moreover, a dotted arrow $(w_1,\cdots, w_n) \dashrightarrow (v_1,\cdots, v_{n+1})$ in $\overline{Q_{B}}^{\M}$ satisfies the word $w_1\cdots w_n =v_1\cdots v_{n+1}.$
Consequently, a zigzag path originating from $(w_1, \dots, w_n)$ to  $(v_1, \dots, v_m)$ in $\overline{Q_{B}}^{\M}$ satisfies the product $w_1 \cdots w_n$ reduces to 	an expression involving $a v_1 \cdots v_m b$ in $A$ in finitely many  steps, for $a,b\in \mathcal{B}.$

We now proceed to compute all even-length zigzag paths ending at the critical vertex $(x,x,z,y)$ in $\overline{Q_B}^{\widetilde{\M}}$. 
 Let $p$ be an even-length zigzag path in $\overline{Q_B}^{\widetilde{\M}}$  ending at $(x,x,z,y)$.  We first determine the starting point of $p.$
	 If the path $p$ contains no dotted arrows that are the reverses of arrows in $\M'$, then it is a zigzag path in $\overline{Q_{B}}^{\M}$.	
	 Consequently, the  start of $p$, denoted by   $(w_1,w_2,w_3,w_4)$, satisfies that the word $w_1w_2w_3w_4$ reduces to an expression involving  $a xxzy b$ in finitely many steps, for $a,b\in \mathcal{B}.$
	 A direct computation shows that only the following elements reduce to  an expression involving   $a xxzy b$  in finitely many steps:
	 $$a' xy^{p_1}x \cdots y^{p_m} xxzy b' \quad \text{and} \quad a' xy^{p_1}x \cdots y^{p_m}xxzxy^{n}x b',$$
	 where  $m\geqslant 0, p_1,\cdots, p_m,n\geqslant 0,$ and $a'x,yb', xb' \in \B$.
	 If $w_1w_2w_3w_4$ is  one of these monomials with $m\geqslant 1,$  then one of $ w_1, w_2,w_3,w_4$ does not lie in $\B_+$,		contradicting the assumption that all $w_i$ are in $\B_+$.
	 Hence $m=0$ and the start of $p$ is either
	 \begin{equation}\label{eq: vertex to xxzy} (ax, x, z, yb) \quad \text{or} \quad (ax,x,zxy^{j}, y^{j'}xb') \end{equation}
	 with $ j,j'\geqslant 0, $ and $ax,yb,xb' \in \B_+. $

Next, assume that the zigzag path $p$ contains dotted arrows from $\M'$, and let the subpath containing the last such dotted arrow be of the form
	 \begin{equation}\label{eq: last dotted from xxzy} (u_1,u_2,u_3,u_4) \dashrightarrow (u_1',u_2',u_3',u_4',u_5') \to (w_1,w_2,w_3,w_4).\end{equation}
	 Therefore, the subpath $p'$ of $p$ that goes from $ (w_1,w_2,w_3,w_4)$ to the final vertex $(x,x,z,y)$ of $p$ is a zigzag path in $\overline{Q_{B}}^{\M}$.
	 	Applying the  discussion in the previous paragraph to $p'$,  we find that  $(w_1,w_2,w_3,w_4)$ is one of the vertices in \eqref{eq: vertex to xxzy}.
 Now, since the dotted arrow in \eqref{eq: last dotted from xxzy} is from $\M'$, it must have the form
$$    (x,yx,z,y) \dashrightarrow  (x,x,x,z,y) .     $$
	 Inductively,  the first dotted arrow from $\M'$ in the path $p$ has the form $  (x,y^{n}x, z,y) \dashrightarrow (x,x,y^{n-1}x,z,y)    $ with $n \geqslant 1.$
	 	Thus, there is a zigzag path in $\overline{Q_{B}}^{\M}$ from the start of $p$  to $ (x,y^{n}x, z,y)$.
	    By an argument analogous to that in the preceding paragraph, we conclude that the starting vertex of $p$ is either
	     \begin{equation}\label{eq: start vertex to xxzy contain M'}   (axy^{i}, y^{i'}x, z, yb) \quad \text{or} \quad   (ax y^{i}, y^{i'} x,zxy^{j}, y^{j'}xb') \end{equation}
	    with $i,i',j,j'\geqslant 0 $ and $ax,yb,xb' \in \B_+. $

Analogously to the $(x, x, z, y)$ case, we conclude that: 
\begin{enumerate}[(\roman*)]
  \item for even-length zigzag paths ending at $(x, x, z)$, the only possible starting points are     $(axy^{i}, y^{i'}x, zb)$, where $i, i' \geqslant 0$ and $ax, z b \in \B_+;$
  \item for those ending at $(x, z, y)$, the only possible starting points are   $ (ax,z,yb) $ and $(ax, zxy^{j}, y^{j'}xb')$, where  $j, j'\geqslant 0 $ and $ax,yb,xb' \in \B_+; $
  \item  for those ending at $(x, x)$, the only possible starting points are $ (axy^{i}, y^{i'}xb),  $  where  $i, i'\geqslant 0 $ and $ax, xb  \in \B_+; $
  \item  for those ending at $(x, z)$, the only possible starting points are $ (ax, zb),  $  where  $ax, zb  \in \B_+; $
  \item  for those ending at $(z, y)$, the only possible starting points are $ (azxy^{j}, y^{j'}x b)   $ and  $ (az, yb'),$  where  $j, j'\geqslant 0 $ and $az, xb, yb' \in \B_+. $
\end{enumerate}

Finally, we enumerate all computed non-trivial even-length zigzag paths that terminate at the critical vertices. 

Using the notation $a = a_1 a_2 \cdots a_n \in \B$ where each $a_i \in \{x,y,z\}$, the paths from $a$ to vertices in $V_1^{\widetilde{\M}}$ are given, for each $0\leqslant i \leqslant n-1$, as follows: 
 $$ \begin{tikzcd}[row sep = 0.5em,column sep =1em]
	 		& & a \ar[lld, dashrightarrow,   "d_{2}^{-1}{,} 1\otimes  1"' near end] \ar[bend left = 90, start anchor={[xshift=15pt]}, dd,above, no head, "\text{pattern repeats } i \text{ times}" ] \\
	 	 (a_1, a_2\cdots a_n)  \ar[rrd,   "d_{2}^{0}{,} a_1\otimes  1"near end ]&  & \\
	 		& 	 & a_2\cdots a_n \ar[dd, dotted, no head, thick,  "\text{weight} : a_2\cdots a_i \ot 1 " ] \\
            & & \\
	 	 	& 	 & a_{i+1}\cdots a_n  \ar[lld, dashrightarrow,   "d_{2}^{-1}{,} 1\otimes  1"' near end] \\
        (a_{i+1}, a_{i+2}\cdots a_n) \ar[rrd,   "d_{2}^{2}{,} 1\otimes  a_{i+2}\cdots a_n"near end ] & & \\
           &     &  a_{i+1}; 
	 	\end{tikzcd}$$
the paths from $ (axy^{i}, y^{i'}xb) $ to $(x,x)$ are given, for each $0\leqslant k \leqslant i+i'$, as follows:  
 $$ \begin{tikzcd}[row sep = 0.5em,column sep =1em]
	 		& & (axy^{i}, y^{i'}xb) \ar[lld, dashrightarrow,   "d_{3}^{-1}{,} 1\otimes  1"' near end] \ar[bend left = 90, start anchor={[xshift=15pt]}, dd,above, no head, "\text{pattern repeats } n \text{ times}" ] \\
	 	 (a_1, a_2\cdots a_nxy^{i}, y^{i'}xb)  \ar[rrd,   "d_{3}^{0}{,} a_1\otimes  1"near end ]&  & \\
	 		& 	 & (a_2\cdots a_nxy^{i}, y^{i'}xb) \ar[dd, dotted, no head, thick,  "\text{weight} : a_2\cdots a_n \ot 1 " ] \\
  & &   \\
	 	 	& 	 & (xy^{i}, y^{i'}xb) \ar[lld, dashrightarrow,   "d_{3}^{-1}{,} 1\otimes  1"' near end] \\
         (x, y^{i}, y^{i'}xb) \ar[rrd,   "d_{3}^{2}{,} 1\otimes  1"near end ] & & \\
           	&  \  & (x, y^{i+i'}xb) \ar[lld, dashrightarrow,   "d_{3}^{-2}{,}  - 1\otimes  1"' near end]\\
            (x, y^{i+i'}x, b) \ar[rrd,   "d_{3}^{3}{,} - 1\otimes  b"near end ] & & \\
         	&  \  & (x, y^{i+i'}x)  \ar[lld, dashrightarrow,  "d_{3}^{-2}{,} 1\otimes  1"' near end]    \ar[bend left = 90, dddd,above, no head, "\text{pattern repeats } k \text{ times}" ]\\
	 	(x,x,y^{i+i'-1}x)\ar[rrd, "d_{3}^{1}{,}1 \otimes  1" near end ] 	\ar[red,rrddd, "d_{3}^{0}{,} x \otimes  1" near end ] & &  \\
	 	 &       &  	(yx,y^{i+i'-1}x) \ar[lld, dashrightarrow,  "  d_{3}^{-1}{,} 1\otimes  1 ~"' near end] \\
	 		(y, x,y^{i+i'-1}x) \ar[rrd, "\!\!\!\!\! d_{3}^{0}{,} y \otimes  1" near end ] &   &   \\
	 			  			  & 	   &   (x,y^{i+i'-1}x) \ar[dd, dotted, no head, thick,  "\text{weight} : (x+y)^{k-1} \ot 1 " ]   \\
  & &  \\
	 		 		  & 	\    &   (x,y^{i+i'-k}x)\ar[lld, dashrightarrow,  "  d_{3}^{-2}{,} 1\otimes  1 ~"' near end]  \\
	(x, x,y^{i+i'-k-1}x) \ar[rrd,  "d_{3}^{3}{,} -1 \otimes  y^{i+i'-k-1}x" near end ] &   &   \\
&  &  (x,x);
	 	\end{tikzcd}$$
the paths from $ (ax, zb) $ to $(x,z)$ are given as follows:  
 $$ \begin{tikzcd}[row sep = 0.5em,column sep = 1em]
	 		& & (ax,zb) \ar[lld, dashrightarrow,   "d_{3}^{-1}{,} 1\otimes  1"' near end] \ar[bend left = 90, start anchor={[xshift=15pt]}, dd,above, no head, "\text{pattern repeats } n \text{ times}" ] \\
	 	 (a_1, a_2\cdots a_nx,zb)  \ar[rrd,   "d_{3}^{0}{,} a_1\otimes  1"near end ]&  & \\
	 		& 	 & (a_2\cdots a_nx,zb) \ar[dd, dotted, no head, thick,  "\text{weight} : a_2\cdots a_n \ot 1 " ] \\
  &  &  \\ 
	 	 	& 	 & (x,zb)  \ar[lld, dashrightarrow,   "d_{3}^{-2}{,} - 1\otimes  1"' near end] \\
       (x,z,b)\ar[rrd,   "d_{3}^{3}{,} -1\otimes  b"near end ] & & \\
           &     &  (x,z); 
	 	\end{tikzcd}$$
the paths from $ (az, yb') $ to $(z,y)$ are given as follows:  
 $$ \begin{tikzcd}[row sep = 0.5em,column sep = 1em]
	 		& & (az,yb') \ar[lld, dashrightarrow,   "d_{3}^{-1}{,} 1\otimes  1"' near end] \ar[bend left = 90, start anchor={[xshift=15pt]}, dd,above, no head, "\text{pattern repeats } n \text{ times}" ] \\
	 	 (a_1, a_2\cdots a_nz,yb')  \ar[rrd,   "d_{3}^{0}{,} a_1\otimes  1"near end ]&  & \\
	 		& 	 & (a_2\cdots a_nz,yb') \ar[dd, dotted, no head, thick,  "\text{weight} : a_2\cdots a_n \ot 1 " ] \\
 &  &  \\
	 	 	& 	 & (z,yb')  \ar[lld, dashrightarrow,   "d_{3}^{-2}{,} - 1\otimes  1"' near end] \\
       (z,y,b')\ar[rrd,   "d_{3}^{3}{,} -1\otimes  b'"near end ] & & \\
           &     &  (z,y); 
	 	\end{tikzcd}$$
the paths from $(azxy^{j}, y^{j'}x b) $ to $(z,y)$ are given as follows:  
 $$ \begin{tikzcd}[row sep = 0.5em,column sep = 1em]
	 		& & (azxy^{j}, y^{j'}x b) \ar[lld, dashrightarrow,   "d_{3}^{-1}{,} 1\otimes  1"' near end] \ar[bend left = 90, start anchor={[xshift=15pt]}, dd,above, no head, "\text{pattern repeats } n \text{ times}" ] \\
	 	 (a_1, a_2\cdots a_nzxy^{j}, y^{j'}xb)  \ar[rrd,   "d_{3}^{0}{,} a_1\otimes  1"near end ]&  & \\
	 		& 	 & (a_2\cdots a_nzxy^{j}, y^{j'}xb) \ar[dd, dotted, no head, thick,  "\text{weight} : a_2\cdots a_n \ot 1 " ] \\
  &  & \\
	 	 	& 	 & (zxy^{j}, y^{j'}xb) \ar[lld, dashrightarrow,   "d_{3}^{-1}{,} 1\otimes  1"' near end] \\
         (z, xy^{j}, y^{j'}xb) \ar[rrd,   "d_{3}^{2}{,} -1\otimes  1"near end ] & & \\
           	&  \  & (z, y^{j+j'+1}xb) \ar[lld, dashrightarrow,   "d_{3}^{-2}{,}  - 1\otimes  1"' near end]\\
            (z, y, y^{j+j'}x b) \ar[rrd,   "d_{3}^{3}{,} - 1\otimes y^{j+j'}x b"near end ] & & \\
         	&  \  & (z,y);
	 	\end{tikzcd}$$
the paths from $(axy^{i}, y^{i'}x, zb) $ to $(x,x,z)$ are given as follows:  
 $$ \begin{tikzcd}[row sep = 0.5em,column sep = 1em]
	 		& & (axy^{i}, y^{i'}x, zb) \ar[lld, dashrightarrow,   "d_{4}^{-1}{,} 1\otimes  1"' near end] \ar[bend left = 90, start anchor={[xshift=15pt]}, dd,above, no head, "\text{pattern repeats } n \text{ times}" ] \\
	 	 (a_1, a_2\cdots a_n xy^{i}, y^{i'}x, zb)  \ar[rrd,   "d_{4}^{0}{,} a_1\otimes  1"near end ]&  & \\
	 		& 	 & (a_2\cdots a_nxy^{i}, y^{i'}x, zb)\ar[dd, dotted, no head, thick,  "\text{weight} :a_2\cdots a_n \ot 1 " ] \\
 & & \\
&\  &  	( xy^{i}, y^{i'}x, zb) \ar[lld, dashrightarrow,  "d_{4}^{-1}{,} 1\otimes  1"' near end] \\
	 		(x,y^{i}, y^{i'}x, zb)\ar[rrd,  "d_{4}^{2}{,} 1 \otimes  1" near end ]  & \  &  \\
	 		&  \  & 	(x, y^{i+i'}x, zb)   \ar[lld, dashrightarrow,  "d_{4}^{-3}{,}   1\otimes  1"' near end]\\
(x, y^{i+i'}x, z,b) \ar[rrd,  "d_{4}^{4}{,}  1 \otimes  b" near end ] & \ & \\
	 		&  \  & 	(x, y^{i+i'}x, z)  \ar[lld, dashrightarrow,  "d_{4}^{-2}{,} 1\otimes  1"' near end]    \ar[bend left = 90, dddd,above, no head, "\text{pattern repeats } i+i' \text{ times}" ]\\
	 	(x,x,y^{i+i'-1}x,z)\ar[rrd, "d_{4}^{1}{,}1 \otimes  1" near end ] 	\ar[red,rrddd, "d_{4}^{0}{,} x \otimes  1" near end ] & &  \\
	 	 &       &  	(yx,y^{i+i'-1}x,z) \ar[lld, dashrightarrow,  "  d_{4}^{-1}{,} 1\otimes  1 ~"' near end] \\
	 		(y, x,y^{i+i'-1}x,z) \ar[rrd, "\!\!\!\!\! d_{4}^{0}{,} y \otimes  1" near end ] &   &   \\
	 			  			  & 	   &   (x,y^{i+i'-1}x,z) \ar[dd, dotted, no head, thick,  "\text{weight} : (x+y)^{i+i'-1} \ot 1 " ]   \\
& & \\
	 		 		  & 	\    &   (x,x,z)  \\
	 	\end{tikzcd}$$
the paths from $ (ax,z,yb)$ to $(x,z,y)$ are given as follows:  
 $$ \begin{tikzcd}[row sep = 0.5em,column sep = 1em]
	 		& & (ax,z,yb) \ar[lld, dashrightarrow,   "d_{4}^{-1}{,} 1\otimes  1"' near end] \ar[bend left = 90, start anchor={[xshift=15pt]}, dd,above, no head, "\text{pattern repeats } n \text{ times}" ] \\
	 	 (a_1, a_2\cdots a_nx,z,yb)  \ar[rrd,   "d_{4}^{0}{,} a_1\otimes  1"near end ]&  & \\
	 		& 	 & (a_2\cdots a_nx,z,yb) \ar[dd, dotted, no head, thick,  "\text{weight} : a_2\cdots a_n \ot 1 " ] \\
& & \\
	 	 	& 	 & (x,z,yb)  \ar[lld, dashrightarrow,   "d_{4}^{-3}{,}   1\otimes  1"' near end] \\
       (x,z,y,b)\ar[rrd,   "d_{4}^{4}{,}  1\otimes  b"near end ] & & \\
           &     &  (x,z,y); 
	 	\end{tikzcd}$$
the paths from $(ax, zxy^{j}, y^{j'}xb')$ to $(x,z,y)$ are given as follows:  
 $$ \begin{tikzcd}[row sep = 0.5 em,column sep = 1em]
	 		& & (ax, zxy^{j}, y^{j'}xb') \ar[lld, dashrightarrow,   "d_{4}^{-1}{,} 1\otimes  1"' near end] \ar[bend left = 90, start anchor={[xshift=15pt]}, dd,above, no head, "\text{pattern repeats } n \text{ times}" ] \\
	 	 (a_1, a_2\cdots a_nx,zxy^{j}, y^{j'}xb')  \ar[rrd,   "d_{4}^{0}{,} a_1\otimes  1"near end ]&  & \\
	 		& 	 & (a_2\cdots a_nx,zxy^{j}, y^{j'}xb') \ar[dd, dotted, no head, thick,  "\text{weight} : a_2\cdots a_n \ot 1 " ] \\
 &  &  \\
	 	 	& 	 & (x,zxy^{j}, y^{j'}xb')  \ar[lld, dashrightarrow,   "d_{4}^{-2}{,}   - 1\otimes  1"' near end] \\
       (x,z, xy^{j}, y^{j'}xb')\ar[rrd,   "d_{4}^{3}{,}  1\otimes 1"near end ] & & \\
           &     &   (x,z,y^{j+j'+1}xb') \ar[lld, dashrightarrow,   "d_{4}^{-3}{,}   1\otimes  1"' near end] \\
           (x,z,y,y^{j+j'}xb')\ar[rrd,   "d_{4}^{4}{,}  1\otimes y^{j+j'}xb'"near end ]  & & \\ 
            &  &  (x,z,y);
	 	\end{tikzcd}$$
the paths from $(a x y^{i}, y^{i'} x, z, y b)$ to $(x, x, z, y)$ are given as follows:  
	 $$ \begin{tikzcd}[row sep = 0.5em,column sep = 0.5em]
	 		& & (axy^{i}, y^{i'}x, z, yb) \ar[lld, dashrightarrow,   "d_{5}^{-1}{,} 1\otimes  1"' near end] \ar[bend left = 90, start anchor={[xshift=15pt]}, dd,above, no head, "\substack{\text{pattern repeats } \\ n \text{ times}}" ] \\
	 	 (a_1, a_2\cdots a_n xy^{i}, y^{i'}x, z, yb)  \ar[rrd,   "d_{5}^{0}{,} a_1\otimes  1"near end ]&  & \\
	 		& 	 & (a_2\cdots a_nxy^{i}, y^{i'}x, z, yb)\ar[dd, dotted, no head, thick,  "\text{weight} :a_2\cdots a_n \ot 1 " ] \\
&  &  \\
&\  &  	( xy^{i}, y^{i'}x, z, yb) \ar[lld, dashrightarrow,  "d_{5}^{-1}{,} 1\otimes  1"' near end] \\
	 		(x,y^{i}, y^{i'}x, z, yb)\ar[rrd,  "d_{5}^{2}{,} 1 \otimes  1" near end ]  & \  &  \\
	 		&  \  & 	(x, y^{i+i'}x, z, yb)   \ar[lld, dashrightarrow,  "d_{5}^{-4}{,} - 1\otimes  1"' near end]\\
(x, y^{i+i'}x, z, y,b) \ar[rrd,  "d_{5}^{5}{,} -1 \otimes  b" near end ] & \ & \\
	 		&  \  & 	(x, y^{i+i'}x, z, y)  \ar[lld, dashrightarrow,  "d_{5}^{-2}{,} 1\otimes  1"' near end]    \ar[bend left = 90, dddd,above, no head, "\substack{\text{pattern repeats }\\ i+i' \text{ times}}" ]\\
	 	(x,x,y^{i+i'-1}x,z,y)\ar[rrd, "d_{5}^{1}{,}1 \otimes  1" near end ] 	\ar[red,rrddd, "d_{5}^{0}{,} x \otimes  1" near end ] & &  \\
	 	 &       &  	(yx,y^{i+i'-1}x,z,y) \ar[lld, dashrightarrow,  "  d_{5}^{-1}{,} 1\otimes  1 ~"' near end] \\
	 		(y, x,y^{i+i'-1}x,z,y) \ar[rrd, "\!\!\!\!\! d_{5}^{0}{,} y \otimes  1" near end ] &   &   \\
	 			  			  & 	   &   (x,y^{i+i'-1}x,z,y) \ar[dd, dotted, no head, thick,  "\text{weight} : (x+y)^{i+i'-1} \ot 1 " ]   \\
 &  &  \\
	 		 		  & 	\    &   (x,x,z,y);  \\
	 	\end{tikzcd}$$
 the paths from $(axy^{i}, y^{i'}x, zxy^{j}, y^{j'}xb')$ to $(x, x, z, y)$  are given as follows:  
$$	 	\begin{tikzcd}[row sep = 0.5em,column sep = 0.5em]
	 		& & (axy^{i}, y^{i'}x, zxy^{j}, y^{j'}xb') \ar[lld, dashrightarrow, "d_{5}^{-1}{,} 1\otimes  1"' near end]\ar[bend left = 90, start anchor={[xshift=15pt]}, dd,above, no head, "\substack{\text{pattern } \\ \text{repeats } \\ n \text{ times}}" ] \\
	 		(a_1,  a_2\cdots a_n xy^{i}, y^{i'}x, zxy^{j}, y^{j'}xb')  \ar[rrd,   "d_{5}^{0}{,} a_1\otimes  1"near end ]&  & \\
	 		&   & 	( a_2\cdots a_n xy^{i}, y^{i'}x, zxy^{j}, y^{j'}xb')	\ar[dd, dotted, no head, thick,  "\text{weight} : a_2\cdots a_n \ot 1 " ]   \\
& & \\
	 		&\  &  		(xy^{i}, y^{i'}x, zxy^{j}, y^{j'}xb') \ar[lld, dashrightarrow,  "d_{5}^{-1}{,} 1\otimes  1"' near end] \\
	 		(x, y^{i}, y^{i'}x, zxy^{j}, y^{j'}xb') \ar[rrd,  "d_{5}^{2}{,} 1 \otimes  1" near end ]  & \  &  \\
	 		&  \  & 	(x,y^{i+i'}x, zxy^{j}, y^{j'}xb')   \ar[lld, dashrightarrow,  "d_{5}^{-3}{,} 1\otimes  1"' near end]\\
	 		(x,y^{i+i'}x, z, xy^{j}, y^{j'}xb') \ar[rrd,  "d_{5}^{4}{,} -1 \otimes  1" near end ]  & \  &  \\
	 		&  \  & 	(x,y^{i+i'}x, z, y^{j+j'+1}xb')   \ar[lld, dashrightarrow,  "d_{5}^{-4}{,} -1\otimes  1"' near end]\\
	 		(x,y^{i+i'}x, z, y,  y^{j+j'}xb') \ar[rrd,  "d_{5}^{5}{,} -1 \otimes   y^{j+j'}xb'" near end ]  & \  &  \\
	 		&  \  & 	(x,y^{i+i'}x, z,y)   \ar[lld, dashrightarrow,  "d_{5}^{-2}{,} 1\otimes  1"' near end]   \ar[bend left = 90, dddd,above, no head, "\substack{\text{pattern} \\\text{repeats } \\ i+i' \text{ times}}"]\\
	 	 	(x,x,y^{i+i'-1}x,z,y)\ar[rrd, "d_{5}^{1}{,}1 \otimes  1" near end ] 	\ar[red,rrddd, "d_{5}^{0}{,} x \otimes  1" near end ] & &  \\
	 	 &       &  	(yx,y^{i+i'-1}x,z,y) \ar[lld, dashrightarrow,  "  d_{5}^{-1}{,} 1\otimes  1 ~"' near end] \\
	 		(y, x,y^{i+i'-1}x,z,y) \ar[rrd, "\!\!\!\!\! d_{5}^{0}{,} y \otimes  1" near end ] &   &   \\
	 			  			  & 	    &   (x,y^{i+i'-1}x,z,y) \ar[dd, dotted, no head, thick,  "\text{weight} : (x+y)^{i+i'-1} \ot 1 " ]  \\
& & \\
	 			  			  & 	   &   (x,x,z,y).  \\
	 	\end{tikzcd}$$

\end{proof}

\section{The Tamarkin--Tsygan Calculus}\label{Sect: TTCalculus}

This appendix recalls some basic concepts of  Tamarkin--Tsygan calculus, adapting the definitions from   \cite{LZZ16}; see also \cite{W19}.

We begin by recalling the definitions of Gerstenhaber algebra and Tamarkin--Tsygan calculus.

\begin{defn}
 A \textbf{Gerstenhaber algebra} over a field $\bfk$ is a triple $(H^{\bullet},\cup,[-,-]),$ where $H^{\bullet} = \oplus_{n\in \mathbb{Z}} H^{n}$ is a graded space over $\bfk$ equipped with two bilinear maps:
  the cup product of degree $0$
 $$\cup:H^{\bullet} \otimes H^{\bullet}\to H^{\bullet}$$
and a Lie bracket of degree $-1$
  $$[-,-]:H^{\bullet}\otimes H^{\bullet}\to H^{\bullet}$$ such that
  \begin{itemize}
  \item[(i)] $(H^{\bullet},\cup)$ is a graded commutative algebra;
  \item[(ii)] $(H^{\bullet}[1], [-,-])$ is a graded Lie algebra;
  \item[(iii)] $[-,-]$ satisfies the graded Leibniz rule with respect to $\cup$: for any $f,g,h \in H^{\bullet},$
 $$ [f,g\cup h] = [f,g]\cup h + (-1)^{(|f|+1)|g|}g\cup [f,h].$$
  \end{itemize}
\end{defn}

 \begin{defn}
   A \textbf{Tamarkin--Tsygan calculus} or a \textbf{differential calculus} is  the data  $(H^{\bullet}, \cup, [-,-], H_{\bullet}, \cap, B)$ of $\mathbb{Z}$-graded vector spaces satisfying the following properties:
\begin{itemize}
     \item[(i)] $(H^\bullet,\cup,[-,-])$ is a Gerstenhaber algebra;
     \item[(ii)] $H_\bullet$ is a graded module over    $(H^{\bullet},\cup)$   via the cap product
      $$\cap: H_{m} \ot H^{n}  \to H_{m-n};$$
     \item[(iii)]  the Connes' differential
     $$ B: H_{\bullet} \to H_{\bullet +1} $$
      satisfies   $B^2=0$ and
      $$[[B, i_f]_{\mathrm{gr}}, i_g]_{\mathrm{gr}} = i_{[f,g]},$$
      for $f \in H^{n},g\in H^{m},$ where $i_f : H_{p} \to H_{p -n} $ is defined as $ i_f(x)  = (-1)^{pn}x \cap f $ and $[-,-]_{\mathrm{gr}}$ is the graded commutator, i.e., $[\alpha,\beta] = \alpha \beta -(-1)^{|\alpha||\beta|} \beta \alpha.$
\end{itemize}
 \end{defn}

Let $A$ be an arbitrary  augmented  associative algebra over $\bfk$.
Denote  $\overline{A} = A / \bfk \mathrm{1}_A,$ and let  $A^e : = A \ot A^{\mathrm{op}}$ be the enveloping algebra of $A$.

The Hochschild cohomology  of $A$ is  defined as the graded vector spaces $\HH^\bullet(A):=\Ext^{-\bullet}_{A^e}(A,A),$
where we adopt the convention that its grading is concentrated in \textbf{non-positive} degrees.
 The Hochschild homology  of $A$ is defined as the graded vector spaces $\HH_\bullet(A):=\Tor^{A^e}_\bullet(A,A).$

There is an $A^{e}$-module projective resolution of $A$, called the reduced two-sided bar resolution and denoted by $B(A,A)$, which is defined as follows.
\begin{itemize}
  \item[(i)] For $n\geqslant 0,$ let $B(A,A)_n = A \ot \overline{A}^{\ot n} \ot A;$
  \item[(ii)] the differential $d_n:B(A,A)_n \to B(A,A)_{n-1}$ sends  $a_0\otimes  a_1\otimes \cdots \otimes a_{n}\otimes a_{n+1}$ to
  $$ \sum_{i=0}^{n} (-1)^i a_0\otimes\cdots \otimes a_ia_{i+1}\otimes \cdots \otimes a_{n+1}.$$
\end{itemize}

The following two complexes yield a computational approach to  Hochschild cohomology and homology.
\begin{defn}\label{defn:HHChainComplex}
	\begin{itemize}
		 \item[(i)]  The \textbf{Hochschild cochain complex} of $A$ is defined as $C^\bullet(A):=\Hom_{A^e}(B(A,A)_{\bullet},A)$.   	
		 \item[(ii)]  The \textbf{Hochschild chain complex} of $A$ is defined as the complex $C_\bullet(A) := A \otimes_{A^e} B(A, A)_\bullet$.
	\end{itemize}
\end{defn}
In particular, for $n\geqslant 0,$  we have
$$C^{n}(A) = \Hom_{A^e}(A \ot \overline{A}^{\ot n} \ot A, A) \cong \Hom(\overline{A}^{\ot n}, A)$$
 and
 $$C_{n}(A) = A\ot_{A^{e}} (A \ot \overline{A}^{\ot n} \ot A) \cong A \ot \overline{A}^{\ot n}.$$

Now, we provide the definitions of the cup product, cap product, Gerstenhaber bracket, and Connes' differential for the Hochschild homology and cohomology of $A$.

\begin{prop-def}\cite{G63}\label{defn:CupProduct}
 Let $f\in C^{n}(A), g\in C^{m}(A),$ the \textbf{cup product}  $f\cup g  \in C^{m+n}(A) = \overline{A}^{\ot n}$ is defined as
 $$f\cup g(a_1\ot \cdots a_{m+n}) : = f(a_1\ot \cdots \ot a_{n}) g(a_{n+1}\ot \cdots \ot a_{m+n}).$$
 This cup product induces a well-defined  product on $\HH^{\bullet}(A)$ of degree $0$
 $$\cup: \HH^{-n}(A) \ot \HH^{-m}(A) \to  \HH^{-m-n}(A). $$
\end{prop-def}

\begin{prop-def}\cite{G63}\label{defn:GBracket}
    Let $f\in C^{n}(A)$ and $g\in C^{m}(A).$ For $1\leqslant i \leqslant n,$  define $f\overline{\circ}_i g \in C^{m+n-1}(A)$ as follows.
    \begin{itemize}
      \item[(i)] If $n,m \geqslant 1,$ then
      \begin{align*}
         (f\overline{\circ}_i g) (a_1\ot \cdots \ot  a_{m+n-1})    : =&    f(a_1 \ot \cdots \ot a_{i-1} \ot g(a_{i}\ot \cdots \ot a_{i+m-1})  \\
         &      \ot a_{i+m} \ot \cdots \ot a_{m+n-1});
      \end{align*}
\item[(ii)]  if $n\geqslant 1, m = 0,$ then $g\in A$ and
$$(f\overline{\circ}_i g)(a_1\ot \cdots \ot a_{n-1}) : = f(a_1\ot \cdots \ot a_{i-1} \ot g \ot a_{i} \ot \cdots \ot a_{n-1}); $$
\item[(iii)] for other case, $f\overline{\circ}_i g = 0.$
    \end{itemize}
The \textbf{Gerstenhaber bracket} of $f$ and $g$ is given by
 $$[f,g] : = f\overline{\circ} g -(-1)^{(n-1)(m-1)} g\overline{\circ} f  , $$
 where $f\overline{\circ} g = \sum_{i=1}^{n} (-1)^{(m-1)(i-1)} f\overline{\circ}_{i}g.$
The above $[-,-]$  induces a well-defined  Lie bracket on $\HH^{\bullet}(A)$  of degree $1$
$$ [-,-]: \HH^{-m}(A) \ot \HH^{-n}(A) \to  \HH^{-m-n+1}(A),$$
such that $(\HH^{-\bullet}(A), \cup, [-,-])$ is a Gerstenhaber algebra.
\end{prop-def}

\begin{prop-def}\label{defn:CapProduct}
   Let $f\in C^{n}(A), z = a_0 \ot \overline{a_1}\otimes \cdots \otimes \overline{a_m} \in C_{m}(A)$, the \textbf{cap product} $z\cap f \in C_{m-n}(A)$ is defined as
   $$ z\cap f : =  (-1)^{nm}(a_0 f(\overline{a_1} \otimes \cdots \otimes \overline{a_n})) \otimes \overline{a_{n+1}}\otimes\cdots\otimes \overline{a_m}. $$
   This cap product induces a well-defined map of degree $0$ on the level of homology
   $$\cap : \HH_{m}(A) \ot \HH^{-n}(A) \to \HH_{m-n}(A).$$
\end{prop-def}

\begin{prop-def}\label{defn:CDifferential}
   Let $z = a_0 \ot \overline{a_1}\otimes \cdots \otimes \overline{a_m} \in C_{n}(A),$ the \textbf{Connes' differential} $B(z) \in C_{n+1}(A)$  is defined as
   $$ B(z) : = \sum_{i=0}^n (-1)^{ni} 1\otimes \overline{a_{i+1}}\otimes\cdots\otimes \overline{a_n}\otimes \overline{a_0}\otimes\cdots\otimes \overline{a_i}.$$
It  induces a well-defined map of degree $1$ on the level of homology
$$B: \HH_{n}(A) \to \HH_{n+1}(A).$$
\end{prop-def}

\begin{prop} \cite{TT00}
    The data $(\HH_\bullet(A),\cup,[-,-], \HH^{-\bullet}(A), \cap, B)$  is a Tamarkin--Tsygan calculus.
\end{prop}

\end{document}